\documentclass[11pt, oneside]{amsart}

\usepackage{amsaddr}
\title{On the geometric trace of a generalized Selberg trace formula}
\author{Andr\'as Bir\'o}
\address{Alfr\'ed R\'enyi Institute of Mathematics,\\
Budapest, Hungary\\}
\email{biro.andras@renyi.hu}
\author{D\'avid  
T\'oth}
\address{Alfr\'ed R\'enyi Institute of Mathematics,\\
Budapest, Hungary\\and\\Department of Computer Science and Information Theory\\
Budapest University of Technology and Economics}
\email{toth.david@renyi.hu}

\thanks{
The research reported in this paper is supported by the “TKP2020,
National Challenges Program” of the National Research, Development and
Innovation Office (BME NC TKP2020) and by
 the Ministry of Innovation and
Technology and the National Research, Development and Innovation
Office within the Artificial Intelligence National Laboratory of Hungary.
It is also supported by the MTA–RI Lendület "Momentum" Analytic Number Theory and Representation Theory Research Group, by the NKFIH (National Research, Development and Innovation Office) grants FK 135218 (D.T.), K135885 and K 143876 (A.B.) and by the Rényi Intézet Lendület Automorphic Research Group.}

\usepackage[utf8]{inputenc}
\usepackage[a4paper, total={6in, 9in}]{geometry}
\usepackage{hyperref}
\usepackage{amsmath}
\usepackage{amsthm}
\usepackage{amssymb}
\usepackage{indentfirst}
\usepackage{t1enc}
\usepackage{array}
\usepackage[mathscr]{eucal}
\usepackage[utf8]{inputenc} 
\usepackage[T1]{fontenc} 
\usepackage{graphpap}
\usepackage[symbols,nogroupskip,sort=def]{glossaries-extra}
\usepackage{imakeidx}
\usepackage{xcolor}

\theoremstyle{plain}
\newtheorem{thm}{Theorem}[section]

\newtheorem{prop}[thm]{Proposition}
\newtheorem{lemma}[thm]{Lemma}

\newcommand{\N}{\mathbb{N}}
\newcommand{\R}{\mathbb{R}}
\newcommand{\C}{\mathbb{C}}

\renewcommand{\C}{\mathbb{C}}
\newcommand{\Q}{\mathbb{Q}}
\newcommand{\Z}{\mathbb{Z}}

\newcommand{\OO}{\mathcal{O}}
\newcommand{\HH}{\mathbb{H}}

\newcommand{\ds}{\displaystyle}

\newcommand{\eps}{\varepsilon}

\newcommand{\abs}[1]{\left|#1\right|}

\newcommand{\ar}{\begin{array}{c}}
\newcommand{\ra}{\end{array}}
\newcommand{\arr}{\begin{array}{cc}}

\newcommand{\Tr}{\textrm{Tr}}

\newcommand{\im}{\textrm{Im}\,}
\newcommand{\ree}{\mathrm{Re}\,}

\newcommand{\mtx}[4]{\left[\begin{array}{cc}#1 & #2 \\ #3 & #4\end{array}\right]}
\newcommand{\mtxx}[5]{\left[\begin{array}{cc}\ds{#1} & \ds{#2} \\[#5mm] \ds{#3} & \ds{#4}\end{array}\right]}
\newcommand{\mtxskip}[4]{\left[\begin{array}{cc}#1 & #2 \\[2mm] #3 & #4\end{array}\right]}

\newcommand{\tr}{\textrm{tr}\,}
\newcommand{\psl}{PSL(2,\R)^n}

\newcommand{\tht}{\vartheta}

\newcommand{\tk}{\mathbf{t}}
\newcommand{\id}[1]{\mathfrak{#1}}

\begin{document}

\maketitle

\begin{abstract}
A certain generalization of the Selberg trace formula was proved by the first named author in 1999. In this generalization instead of considering the integral of $K(z,z)$ (where $K(z,w)$ is an automorphic kernel function) over the fundamental domain, one considers the integral of $K(z,z)u(z)$, where  $u(z)$ is a fixed automorphic eigenfunction of the Laplace operator. This formula was proved for discrete subgroups of $PSL(2,\R)$, and just as in the case of the classical Selberg trace formula it was obtained by evaluating in two different ways ("geometrically" and "spectrally") the integral of $K(z,z)u(z)$.

In the present paper we work out the geometric side of a further generalization of this generalized trace formula: we consider the case of discrete subgroups of $PSL(2,\R)^n$ where $n>1$. Many new difficulties arise in the case of these groups due to the fact that the classification of conjugacy classes is much more complicated for $n>1$ than in the case $n=1$. 
\end{abstract}

\section{Introduction}

\subsection{The Selberg trace formula and its generalizations}

The Selberg trace formula (introduced by A. Selberg, see \cite{selberg})
is a particularly important tool in the theory of automorphic functions, it has many applications
in different branches of mathematics. 
Briefly speaking,  it is obtained by computing the integral
\[
\Tr K = \int_F K(z,z)\,d\mu(z)
\]
in two different ways ("geometrically" and "spectrally"), where
$d\mu(z)=y^{-2}\,dx\,dy$ is the usual measure on the complex upper half-plane $\HH$, 
$F\subset \HH$ is the fundamental domain of a finite volume Fuchsian group $\Gamma\leq PSL(2,\R)$
acting on $\HH$,
and $K(z,w)$ is an appropriate automorphic kernel function (i.e. a function that is invariant
under the action of $\Gamma$). 

A. Bir\'o obtained a generalization of this formula in \cite{biro1} 
by evaluating the integral
\begin{align}\label{1_1_general_trace_fromula}
\Tr K = \int_F K(z,z)u(z)\,d\mu(z),
\end{align}
where the weight $u(z)$ is a Maass form (i.e. an automorphic eigenfunction of 
the Laplace operator).
Biró applied his
generalization to the the hyperbolic circle problem (see \cite{biro2018local}), and some ideas of \cite{biro1} are applied for cycle integral and triple product identities in \cite{biro2000cycle} and \cite{biro2011relation}.

Selberg's trace formula was developed for a general family of groups, and our
aim is to work out the geometric side of Bir\'o's generalized formula
for discrete subgroups of $PSL(2,\R)^n$ where $n>1$. 
For these, the details of the Selberg trace formula are given in 
\cite{efrat}.

\subsection{Discrete subgroups of 
\texorpdfstring{$PSL(2,\R)^n$}{PSL(2,R)n}}

The main example of the discrete subgroups of $PSL(2,\R)^n$ 
is the Hilbert modular group.
For any totally real finite extension $\Q\leq K$ of degree $n>1$ it is defined
as
\[
\Gamma_K:=\left\{\left(\mtx{ a^{(1)} }{ b^{(1)} }{ c^{(1)} }{ d^{(1)} }, 
\dots, \mtx{ a^{(n)} }{ b^{(n)} }{ c^{(n)} }{ d^{(n)} }\right):\,
\mtx{ a }{ b }{ c }{ d }\in PSL(2,\OO_K)\right\},
\]
where $K^{(1)},\dots,K^{(n)}$ are the different embeddings of $K$ into $\R$, and
the images of an element $a\in K$ by these embeddings  are $a^{(1)},\dots,a^{(n)}$.
As usually, $\OO_{K}$ denotes the ring of integers in $K$.

The generalized trace formula (derived from (\ref{1_1_general_trace_fromula})) 
is fully worked out in the PhD thesis \cite{dissertation} in the special case
when $\Gamma$ is a Hilbert modular group for a \emph{quadratic field} \emph{of class number one}. 
Though a more general situation is handled here, the Hilbert modular groups still play an 
important role in the following. Before specifying this role, we give a short summary of
the main results about discrete subgroups of $PSL(2,\R)^n$.
Many of these are proved in \cite{freitag}.

\subsubsection{Action on the product of upper half-planes.}

The group $PSL(2,\R)$ acts on $\HH$ in the usual way,
if $\gamma = \mtx{a}{b}{c}{d}\in PSL(2,\R)$ and $z\in \HH$, then
\begin{align}\label{1_2_1_action}
\gamma z = \frac{az+b}{cz+d}.
\end{align}
This induces a coordinate-wise action of $\psl$ on the product space $\HH^n$. 
For an element $z\in \HH^n$ we  will use the 
notation $z=(z_1,\dots,z_n)$ where $z_k=x_k+iy_k\in \HH$
with $x_k\in\R$ and $y_k\in\R^+$ ($k=1,\dots,n$).
That is,  if $\gamma=(\gamma^{(1)},\dots, \gamma^{(n)})\in\psl$ and $z\in\HH^n$, then 
$\gamma z=(\gamma^{(1)}z_1,\dots,\gamma^{(n)}z_n)$. It is convenient to represent 
an element $\gamma\in\psl$ as in the one dimensional case, i.e. as a matrix 
$\gamma=\mtx{a}{b}{c}{d}$, whose elements are $n$ dimensional 
row vectors with real coordinates, e.g. $a=(a^{(1)},\dots,a^{(n)})$. Then the action
of $\psl$ on $\HH^n$ can be written formally as in (\ref{1_2_1_action}), 
where the operations are meant to be performed coordinate-wise.
It is known that $\Gamma\leq PSL(2,\R)^n$ is discrete if and only if it acts discontinuously on
$\HH^n$ (see Proposition 2.1 in \cite{freitag}).

\subsubsection{Irreducible groups}

The groups $\Gamma,\Gamma'\leq PSL(2,\R)^n$ are said to be \emph{strictly commensurable}
if $\Gamma\cap \Gamma'$ has finite index in both $\Gamma$ and $\Gamma'$. They are said to
be \emph{commensurable} if $\Gamma$ is strictly commensurable 
with a conjugate of $\Gamma'$.
A discrete subgroup $\Gamma\leq \psl$ is said to be 
\emph{irreducible} if $\Gamma$ is not
commensurable with any direct product $\Gamma'\times\Gamma''$, where $\Gamma'$ and $\Gamma''$
are  discrete subgroups of some non-trivial groups $G'$ and $G''$, respectively, for which
$\psl=G'\times G''$ holds.
For a discrete subgroup $\Gamma$ of $\psl$ we have the following equivalent conditions
for irreducibility (see the Corollary after Theorem 2 in \cite{shimizu}):
\begin{enumerate}
\item[(i)] $\Gamma$ contains no element $\gamma=(\gamma^{(1)},\dots,\gamma^{(n)})$ such
that $\gamma^{(i)}=1$ holds for some $i$ while $\gamma^{(j)}\neq 1$ holds for some $j$,
\item[(ii)] there exist no partial product $G'$ of $\psl$ such that
the projection of $\Gamma$ to $G'$ is discrete,
\item[(iii)] for every $\gamma\in\Gamma\setminus\{1\}$ the centralizer of $\gamma$ in $\psl$
is commutative.
\end{enumerate}
In the following $\Gamma$ always denotes an  irreducible subgroup.  
Note that 
property (i) above implies that the Hilbert modular group is irreducible.

\subsubsection{Cusps.}

The group $\psl$ and hence its subgroup  $\Gamma$  act on the set $(\R\cup\{\infty\})^n$.
This action is also given by (\ref{1_2_1_action}) (using the usual extended operations
on the set $\R\cup \{\infty\}$).
 Roughly speaking, an element of $(\R\cup\{\infty\})^n$
is a \emph{cusp} if its stabilizer in $\Gamma$ is in some sense as large as possible.

First, let us consider the element $\infty=(\infty,\dots,\infty)$.
The stabilizer $\Gamma_\infty$ of $\infty$  contains elements 
of the form 
\[
\gamma=\mtx{a}{b}{0}{d},
\]
and
among them there are those for which all the coordinates of the vectors $a$ and $d$ 
are $1$, i.e. the translations. Let us define
\[
\tk_\infty=\tk_\infty(\Gamma)=\left\{b\in\R^n:\mtx{1}{b}{0}{1}\in\Gamma\right\},
\]
then, since $\Gamma$ is discrete, $\tk_\infty$ is a discrete subgroup of $\R^n$,
and hence it is isomorphic to
$\Z^m$ for some $1\leq m\leq n$. The important cases are those, for which $m=n$ holds 
(i.e. $\tk_\infty$ is a 
lattice).

For a general element $\gamma\in\Gamma_\infty$ we have 
$d^{(k)}=(a^{(k)})^{-1}$ ($1\leq k\leq n$),
and therefore  $\gamma z=a^2z+ab$ holds for any $z\in\HH^n$  (again, the operations
are accomplished coordinate-wise).
Notice that the the coordinates of the vector $a^2$ are all positive.
The vectors with this property are called \emph{totally positive}\index{totally positive vectors}
and they form the group $(\R^+)^n$ w.r.t. the coordinate-wise multiplication.
An element $\eps\in(\R^+)^n$ is called a \emph{multiplier}\index{multiplier}
for $\Gamma$ if there is a $\gamma\in\Gamma$ for which $\gamma z=\eps z+b$ holds
for some $b\in\R^n$. 
The multipliers   form a subgroup of 
$(\R^+)^n$, it  is denoted by $\Lambda_\infty=\Lambda_\infty(\Gamma)$. 

If $\tk_\infty$ is a lattice, then $\Lambda_\infty$ 
is a discrete subgroup of $(\R^+)^n$ and
for each $\eps\in\Lambda_\infty$ we have $\eps^{(1)}\cdots\eps^{(n)}=1$ (see Remark 2.3 in \cite{freitag}).
Taking the logarithm coordinate-wise, we obtain that
$\log \Lambda_\infty$ is a discrete subgroup of $\R^n$ 
contained
in the hyperplane
\begin{align}\label{1_2_3_hyperplane}
V=\{a\in\R^n:\,a^{(1)}+\dots+a^{(n)}=0\},
\end{align}
and hence cannot be a lattice. 
That is, if $\tk_\infty\cong \Z^n$, then  $\Lambda_\infty\cong \Z^m$ for some
 $0\leq m \leq n-1$.
  We say that $\infty$ is a \emph{cusp} for $\Gamma$ if $\tk_\infty\cong \Z^n$ and 
$\Lambda_\infty\cong \Z^{n-1}$.

Note that
$\infty$ is a cusp for the Hilbert modular group.
Indeed, in this case the $n$ different
embeddings of $K$ induce a map from $\OO_K$ to $\R^n$ 
whose image is well-known to be a lattice that 
$\tk_\infty$ is isomorphic to.
Also, $\Lambda_\infty$ consists of the squares of the units of $\OO_K$. 
By Dirichlet's unit theorem we have that $\OO_K^\times\cong \Z^{n-1}\times \{\pm 1\}$,
hence $\Lambda_\infty\cong \Z^{n-1}$ and our claim follow.

We say 
that an element $\kappa\in(\R\cup\{\infty\})^n$ is a \emph{cusp}
for the discrete group $\Gamma\leq PSL(2,\R)^n$ if $\infty$ is a cusp
for the group $\sigma^{-1} \Gamma\sigma$ for some $\sigma\in PSL(2,\R)^n$ with
$\sigma\infty=\kappa$.
It is not hard to see that in this case $\infty$ is a cusp for $\rho^{-1}\Gamma\rho$ for every element 
$\rho\in PSL(2,\R)^n$ with $\sigma\infty = \rho\infty=\kappa$. Moreover, 
we have in fact 
$\Lambda_\infty(\sigma^{-1}\Gamma\sigma)=\Lambda_\infty(\rho^{-1}\Gamma\rho)=:\Lambda_\kappa$,
hence this group is determined uniquely by the cusp $\kappa$ and it is called the
\emph{multiplier group for $\kappa$}. Note that $\tk_\infty(\sigma^{-1}\Gamma\sigma)$
is determined only up to a coordinate-wise scalar multiple. Having this in mind, we will use the notation
$\tk_\kappa$ for such a lattice.

\subsubsection{Fundamental domain}

Let $X$ be a locally compact second-countable topological space, and assume that $\Gamma$ is a 
group of topological (i.e. bijective, continuous and 
open) maps of $X$ onto itself. 
The set $F\subset X$ is called a \emph{fundamental set} for $\Gamma$ if
\begin{align}\label{1_2_4_fundamental_set}
X=\bigcup_{\gamma\in\Gamma}\gamma(F).
\end{align}
For example the space $X$ itself is always a fundamental set. Of course 
we are interested in fundamental sets
that are in some sense as small as possible.

Also, we would like to consider measurable sets, so assume further that a 
$\Gamma$-invariant Radon measure $\mu$ is given on $X$. Then a measurable set $F\subset X$ is called a 
\emph{fundamental domain}\index{fundamental domain} for $\Gamma$ if (\ref{1_2_4_fundamental_set}) holds,
and there is a set $S\subset F$ such that $\mu(S)=0$ and different  points
of $F\setminus S$ are not on the same $\Gamma$-orbit.

Every measurable fundamental set contains a fundamental domain
(see Appendix II in \cite{freitag}), hence
there exists a fundamental domain $F\subset \HH^n$ 
for every 
irreducible 
discrete subgroup $\Gamma\leq\psl$.
In the following we always assume that the volume of $F$ is finite.
This holds of course when the quotient space $\HH^n\setminus \Gamma$ 
is compact. However, this quotient is not compact once there are cusps for $\Gamma$
since in that case $F$ contains parts that stretch out to the boundary of $\HH^n$.
Now we are in the position to state the following important theorem
(see Theorem I.1.5 in \cite{efrat}):
\begin{thm}\label{1_3_commensurable_thm}
Assume that $n\geq 2$ and $\Gamma\leq \psl$ is an irreducible discrete subgroup whose 
fundamental domain has finite volume but the quotient space $\HH^n\setminus\Gamma$ 
is not compact. Then $\Gamma$ is commensurable with a Hilbert modular group for a number field 
$K$ of degree $n$.
\end{thm}
In addition, the field $K$ in the theorem above is generated by the elements of
$\Lambda_\kappa$ for any cusp $\kappa$  
(see the proof of Theorem 4 in \cite{shimizu}).
Note that if $\Gamma$ is commensurable with the Hilbert modular group, then
its fundamental domain is of finite volume. 

From now on, $\Gamma$ always denotes a discrete irreducible subgroup of $\psl$ with a fundamental domain
of finite volume, and we will assume also that $\Gamma$ has at least one cusp. 
The group $\Gamma$ acts on the set of its cusps.
It is well known that the number of the equivalence classes (orbits) of the 
cusps for a Hilbert modular group is the class number of the corresponding field $K$ (see \cite{siegel}, Proposition 20 on p. 188). It follows from this and from the theorem above, that the
number of the equivalence classes of the cusps for $\Gamma$ is finite.

To describe the 
fundamental domain of $\Gamma$ we first define the cusp regions.
We set
\[
U_C:=\{z\in\HH^n:\,Ny>C\},
\]
where $C>0$ is a positive real number and $Na=\prod_{k=1}^na^{(k)}$ for any real 
vector $a=(a^{(1)},\dots,a^{(n)})^T$.
If $\kappa$ is a cusp for $\Gamma$ and  $\sigma\in PSL(2,\R)^n$ is a fixed element
for which $\sigma\infty=\kappa$ holds,
then the sets of the form $\sigma(U_C)$ are called the neighbourhoods of $\kappa$. 
Note that the set of these neighbourhoods does not depend on the choice of $\sigma$.

 The stabilizer $\Gamma_\kappa$ of $\kappa$ acts on the sets of the form $U_\kappa=\sigma(U_{C_\kappa})$ 
 ($C_\kappa\in\R^+$, $\sigma\infty=\kappa$). 
To construct a fundamental domain for $\Gamma_\kappa$ in $U_\kappa$, 
it is sufficient to give a fundamental domain $F'_\kappa$ for 
$\sigma^{-1}\Gamma_\kappa\sigma=(\sigma^{-1}\Gamma\sigma)_\infty$ in $U_{C_\kappa}$, and
then the set $F_\kappa:=\sigma(F_\kappa')$ is a fundamental domain
for $\Gamma_\kappa$ in $U_\kappa$.
Recall that
$\mathbf{t_\kappa}$ is a lattice in $\R^n$ 
and that  $\log\Lambda_\kappa$ has rank $n-1$ and it is contained in the hyperplane
$V$ defined in (\ref{1_2_3_hyperplane}).
Let $P_\kappa$ be a fundamental parallelotope for $\mathbf{t_\kappa}$ in $\R^n$ and
let $Q_\kappa$ be a fundamental parallelotope for $\log\Lambda_\kappa$ 
in the vector space $V$.
It is easy to see that
\[
F'_\kappa=
\left\{
z=x+iy\in U_{C_\kappa}:x\in P_\kappa,\, \log\frac{y}{(Ny)^{1/n}}\in Q_\kappa
\right\}
\]
is a fundamental domain for $\sigma^{-1}\Gamma_\kappa\sigma$ in $U_{C_\kappa}$.

To express $F_\kappa$ in a simple way, we introduce the coordinates at the cusp $\kappa$. For
this we fix a \emph{scaling element} $\sigma_\kappa\in\psl$ such that $\sigma_\kappa\infty=\kappa$, a basis $b_1^\kappa,\dots,b_n^\kappa$  in $\R^n$ such that
\[
P_\kappa=\left\{
b\in\R^n:b=\sum_{k=1}^nt_kb_k^\kappa,\,0\leq t_k<1\textrm{ for }1\leq k\leq n
\right\}
\]
is a fundamental parallelotope for $\mathbf{t}_\kappa$, 
and a basis $a_1^\kappa,\dots,a_{n-1}^\kappa$  in $V$ such that
\[
Q_\kappa=\left\{
a\in V:a=\sum_{l=1}^{n-1}s_l a_l^\kappa,\,0\leq s_l<1\textrm{ for }1\leq l\leq n-1
\right\}
\]
is a fundamental parallelotope for $\log\Lambda_\kappa$.
Note that $a_l^{\kappa}=\log \eps_l^{\kappa}$,
where $\eps_1^{\kappa},\dots,\eps_{n-1}^\kappa$ generate the group $\Lambda_\kappa$.
Then for any $z\in\HH^n$ we write $\sigma_\kappa^{-1} z=x'+iy'$ and 
\[
\begin{array}{c}
x'=X_1^\kappa(z) b_1^\kappa+\dots+X_n^\kappa(z) b_n^\kappa,\\
Y_0^\kappa(z)=Ny',\quad
\log \dfrac{y'}{(Ny')^{1/n}}=Y_1^\kappa(z)\log \eps_1^\kappa+\dots+Y_{n-1}^\kappa(z)\log\eps_{n-1}^\kappa.
\end{array}
\]
Once the element $\sigma_\kappa$ and the bases above are fixed, the numbers $X_k^\kappa(z)$ and $Y_l^\kappa(z)$ are uniquely determined and called
the \emph{coordinates of $z$ at the cusp $\kappa$}.  We may simply write 
$X_k^\kappa$ and $Y_l^\kappa$.
Note that we also use the notation
$Y_0(z)=Ny$ or simply $Y_0$ 
(which is the same as $Y_0^\infty(z)$ above once $\infty$ is a cusp and 
we choose $\sigma_\infty=id$).
The fundamental domain of $\Gamma_\kappa$ in $U_\kappa$ can be expressed 
in a simple form in terms of these coordinates:
\[
F_\kappa=\{z\in U_\kappa:\, 
            0\leq X_1^\kappa,\dots,X_n^\kappa<1,\,
            0\leq Y_1^\kappa,\dots,Y_{n-1}^\kappa<1
        \}.
\]

If $\kappa$ and $\kappa'$ are inequivalent cusps of $\Gamma$, 
then there exists neighbourhoods $U$ and $U'$
of $\kappa$ and $\kappa'$, respectively, such that 
$\gamma(U)\cap U'= \emptyset$ holds for any  $\gamma\in\Gamma$ (see Lemma $2.9_2$ in 
\cite{freitag}). Hence, if we fix a maximal set 
$\mathcal{S}$ of 
$\Gamma$-inequivalent cusps, then a real number $C>0$ can be chosen
such that the sets $U_\kappa=\sigma_\kappa(U_C)$ are pairwise disjoint for the cusps in $\mathcal{S}$, 
and the corresponding sets $F_\kappa$ contain at most $1$ point
from every $\Gamma$ orbit.
Finally, the fundamental domain for $\Gamma$ is given in the form 
\[
F=F_0\cup\left(\bigcup_{\kappa\in\mathcal{S}}F_\kappa\right),
\]
where $F_0\subset \HH^n$ is compact.

\subsubsection{Classification of the elements of $\Gamma$}

Recall that an element $\mathrm{id}\neq\gamma\in PSL(2,\R)$ is called 
elliptic, parabolic or hyperbolic, if $\abs{\tr \gamma}<2$,
$\abs{\tr \gamma}=2$ or $\abs{\tr \gamma}>2$, respectively.
An element of $\Gamma$ is called \emph{totally elliptic}
or \emph{totally parabolic},
if each of its components are elliptic or parabolic, respectively.
If there are elements of different types among the components, then this element is called 
\emph{mixed}. Note that if one component of an element is parabolic, then so are the others
by Theorem \ref{1_3_commensurable_thm}. Hence a mixed element consists of elliptic and hyperbolic components.

Before we turn to the case when every component is hyperbolic we examine
the fixed points of the elements. A totally elliptic element has a single fixed point 
$z\in \HH^n$. Since $\Gamma$ acts discontinuously on $\HH^n$, $z$ has a neighborhood $U$
such that the set $\{\gamma\in\Gamma:\, \gamma U\cap U\neq\emptyset\}$ is finite.
This means that a totally elliptic element must be of finite order. A totally parabolic
element fixes a single point in $(\R\cup\{\infty\})^n$.
Since $\Gamma$ is irreducible, the parabolic fixed points are exactly the cusps  
of $\Gamma$ (see Theorem 3 in \cite{shimizu}).
A mixed element with $1\leq m<n$ hyperbolic components fixes $2^{m}$ points in $(\HH\cup\R\cup\{\infty\})^n$.
If every component of $\gamma\in\Gamma$ 
is hyperbolic, then $\gamma$ fixes  $2^n$ points in $(\R\cup\{\infty\})^n$.
Such an element is called \emph{hyperbolic-parabolic}
if  there is cusp among its fixed points. 
Otherwise 
it is called 
\emph{totally hyperbolic}.

\subsection{Fourier expansion of automorphic forms}\label{Fourier_expansion_section}
A function $f:\HH^n\to\C$ is called an \emph{automorphic function}\index{automorphic function} 
with respect to the 
group $\Gamma$ 
if it is invariant under the action of $\Gamma$, that is, $f(\gamma z)=f(z)$ holds for every
$z\in\HH^n$ and $\gamma\in\Gamma$.
An \emph{automorphic form}\index{automorphic form} $u$ is a smooth automorphic function which is 
an eigenfunction of the Laplace operators
\[
\Delta_k=y_k^2
\left(
\frac{\partial^2}{\partial x_k^2}+
\frac{\partial^2}{\partial y_k^2}
\right),\qquad(k=1,\dots,n),
\]
that is, for which the equations
$(\Delta_k+\lambda_k)u=0$
hold with some $\lambda_k\in\C$.
We write these eigenvalues in the form $\lambda_k=s_k(1-s_k)$ for some $s_k\in\C$.

If $u$ is an automorphic form 
and $\kappa$ is a cusp, then $u(\sigma_{\kappa}z)$ is invariant under the action of the translation
operator $T_\alpha u=u(z_1+\alpha_{1},\dots, z_n+\alpha_{n})$ for any $\alpha\in\mathbf{t}_\kappa$,
hence
it has the Fourier expansion
\[
u(z)=\sum_{l\in \mathbf{t}_\kappa^*}\phi(y,l)e^{2\pi i<l,x>},
\]
where $x=(x_1,\dots,x_n)$, $y=(y_1,\dots,y_n)$ and 
$\mathbf{t}_\kappa^*=\{v\in\R^n:\left<v,w\right>\in\Z\,\textrm{ for any }w\in \mathbf{t}_\kappa\}$ is
the \emph{dual lattice} of  $\tk_\kappa$. In general, the dual lattice
is given 
in the following way. 
If $L = A(\Z^n)\subset\R^n$ is a lattice,
where $A\in GL(\R^n)$, then  its dual 
is given by $L^*=(A^{-1})^T(\Z^n)$.
In our case the columns of $A$ are the vectors $b_1^\kappa,\dots,b_n^\kappa$.

For a vector $\alpha\in \R^n$ and a lattice $L \subset \R^n$ we define
\[
\alpha L = \{(\alpha_{1} l_{1},\dots, \alpha_{n}l_{n}):\,(l_{1},\dots,l_{n})\in L\}.
\]
It is easy to see, that if $a\in\tk_\kappa$ and $\eps\in\Lambda_\kappa$,
then $\eps a\in\tk_\kappa$, and as $\Lambda_\kappa$ is a group, we have in fact $\eps\tk_\kappa=\tk_\kappa$. 
If $E$ is the diagonal matrix with the coordinates
of $\eps$ in its diagonal, then  
\begin{align}\label{dual_invar}
\eps\tk_\kappa^*&=E(A^{-1})^T(\Z^n)=E^T(A^{-1})^T(\Z^n)\\
                &=(A^{-1}E)^T(\Z^n)
                =((E^{-1}A)^{-1})^T(\Z^n)=(\eps^{-1}\tk_\kappa)^*=\tk_\kappa^*.\nonumber
\end{align}

Since the Laplace operator commutes with the action of
$\psl$, $u(\sigma_\kappa z)$ is still an eigenfunction of the Laplacians, 
and its
 Fourier coefficients can be expressed by means of
its eigenvalues and the modified Bessel function of the second kind, denoted by $K_\nu(z)$ 
(see Theorem 5.1 in \cite{truelsen}):
\begin{thm}\label{fourier_expansion}
Let $u$ be an automorphic form 
that satisfies
the growth condition
$u(\sigma_\kappa  z)=o(e^{2\pi y_k})$
as $y_k\to \infty$ ($k=1,\dots,n$) (where $\kappa$ is a cusp for $\Gamma$). Then
$u(\sigma_\kappa z)$ admits a Fourier expansion of the form
\begin{equation}\label{fourier_expansion_of_u}
u(\sigma_\kappa z) = \sum_{l\in \tk_\kappa^*}a_\kappa(l,y)e^{2\pi i<l,x>}, 
\end{equation}
where
\[
a_\kappa(l,y) = c_\kappa(l)\sqrt{y_1\dots y_n}K_{s_1-1/2}(2\pi\abs{l_{1}}y_1)\dots K_{s_n-1/2}(2\pi\abs{l_{n}}y_n)
\]
for $l\neq0$, while $a_\kappa(0,y)=:a_\kappa(y)$ is 
the linear combination of two terms of the form 
$y_1^{s_1}\dots y_n^{s_n}$
and  $y_1^{1-s_1}\dots y_n^{1-s_n}$, where the numbers $s_k\in\C$ are such that 
$(\Delta_k +s_k(1-s_k))u=0$.
\end{thm}
In the following we always assume that an automorphic form $u$  satisfies the growth condition
in Theorem \ref{fourier_expansion} and hence admits the Fourier expansion (\ref{fourier_expansion_of_u}). 
Since $u(\sigma_\kappa z)$ remains unchanged if we substitute 
$z\mapsto \eps z$ for any $\eps\in\Lambda_\kappa$, comparing the Fourier coefficients, using
(\ref{dual_invar})
and also that $N\eps=1$ holds,  we obtain that
$c_\kappa({\eps l})=c_\kappa(l)$ for every cusp $\kappa$ and for every $\eps\in\Lambda_\kappa$, 
$l\in \tk_\kappa^*\setminus 0$. 
Also, well-known bounds for the Bessel function 
$K_\nu(z)$
and the absolute convergence of the 
sum in (\ref{fourier_expansion_of_u})
easily imply  the trivial bound 
$c_\kappa(l)\ll e^{\delta\abs{N(l)}^{1/n}}$ for any $\delta>0$,
 where the implied constant depends on
 $\delta$. 
From this the exponential 
decay of 
$u(\sigma_\kappa z)-a_\kappa(y)$
"near the cusp" can
be derived.
Though we will not
detail its 
(technical but straightforward)
proof,
the precise statement is
given in the following
\begin{prop}\label{u_bound_prop}
Let $u$ be an automorphic form
with respect to  $\Gamma$ 
with Laplace eigenvalues 
$s_k(1-s_k)$ that
satisfies 
$u(\sigma_\kappa z)
=o(e^{2\pi y_k})$ for any $1\leq k\leq n$. 
Assume that 
$\log\left(\frac{ y}{(Ny)^{1/n}}\right)$ is bounded,
then 
$u(\sigma_\kappa z)-a_0(y) = 
O(e^{-C N(y)^{1/n}})$ for some constant $C>0$ if 
$N(y)$ is big enough.   
The implied constant depends on 
the  bounds on $N(y)$ and 
$\log\left(\frac{ y}{(Ny)^{1/n}}\right)$. 
\end{prop}

Recall that above we fixed 
the generators $\eps_1^\kappa,\dots,\eps_{n-1}^\kappa$ of $\Lambda_\kappa$.
Their coordinates will be denoted by $(\eps_j^\kappa)^{(k)}$ ($k=1,\dots,n$).
If the zeroth Fourier coefficient of $u(\sigma_\kappa z)$ is non-zero,
then the comparison of them on both sides of the equation
$u(\sigma_\kappa z)=u(\sigma_\kappa (\eps_j^\kappa z))$ gives for each $1\leq j\leq n-1$ (similarly as in section II.1 of 
\cite{efrat}) that
\begin{align}\label{1_3_eps_eq}
\prod_{k=1}^n\left[(\eps_j^\kappa)^{(k)}\right]^{s_k}=1.
\end{align}
Let us define $s:=(s_1+\dots+s_n)/n$,
then by (\ref{1_3_eps_eq}) we have
\begin{equation}\label{eps_s_eqs}
(s_1,\dots,s_n)
\left[
\begin{array}{cccc}
1& \log(\eps_1^\kappa)^{(1)}&\dots &\log (\eps_{n-1}^\kappa)^{(1)} \\
1& \log(\eps_1^\kappa)^{(2)}&\dots &\log (\eps_{n-1}^\kappa)^{(2)} \\
\vdots &\vdots & \ddots &\vdots \\
1& \log(\eps_1^\kappa)^{(n)}&\dots &\log (\eps_{n-1}^\kappa)^{(n)} 
\end{array}
\right]
=(ns,2\pi im_{u,\kappa}^{(1)},\dots,2\pi i m_{u,\kappa}^{(n-1)})
\end{equation}
for some $m_{u,\kappa}=(m_{u,\kappa}^{(1)},\dots, m_{u,\kappa}^{(n-1)})^T\in \Z^{n-1}$. Let us denote the matrix above by 
$\mathcal{E}_\kappa$.
Since the vectors $\log \eps_i^\kappa$ form a basis in the trace $0$ subspace of $\R^n$ and the first
column of $\mathcal{E}_\kappa$ is not in that subspace, we get that $\mathcal{E}_\kappa$ is invertible.
Its inverse is of the form
\[
\mathcal{E}_\kappa^{-1}=
\left[
\begin{array}{cccc}
1/n & 1/n &\dots & 1/n\\
(e_1^\kappa)^{(1)} & (e_1^\kappa)^{(2)} & \dots & (e_1^\kappa)^{(n)} \\
\vdots & \vdots & \ddots & \vdots\\
(e_{n-1}^\kappa)^{(1)} & (e_{n-1}^\kappa)^{(2)} & \dots & (e_{n-1}^\kappa)^{(n)} 
\end{array}
\right],
\]
and
the values $s_1,\dots,s_n$ are determined by $s$ and $m_{u,\kappa}$ through
\begin{align*}
(s_1,\dots,s_n)=(ns,2\pi i m_{u,\kappa}^{(1)},\dots,2\pi i m_{u,\kappa}^{(n-1)})
\mathcal{E}_\kappa^{-1}.
\end{align*}
That is, 
\begin{equation}\label{s_k_def}
s_k=s+\sum_{j=1}^{n-1}2\pi i m_{u,\kappa}^{(j)}(e_j^\kappa)^{(k)},\qquad
y_k^{s_k}=y_k^s\exp\left(\sum_{j=1}^{n-1}2\pi i m_{u,\kappa}^{(j)}(e_j^\kappa)^{(k)}\log y_k \right).
\end{equation}
For a cusp $\kappa$ and for any $m\in \Z^{n-1}$ set 
\begin{align}\label{1_3_grossen_exp_sum}
\lambda_m^\kappa(y)
=\exp\left(\sum_{k=1}^n\sum_{j=1}^{n-1}2\pi i m_j(e_j^\kappa)^{(k)}\log y_k \right)
=\prod_{k=1}^n\prod_{j=1}^{n-1}y_k^{2\pi i m_j(e_j^\kappa)^{(k)}}
\end{align}
for every $y\in(\R^+)^n$.
With this notation  we may write the zeroth coefficient of $u(\sigma_\kappa z)$ in the following way:
\begin{align*}
a_\kappa(y)=
\eta_\kappa (y_1\dots y_n)^s\lambda_{m_{u,\kappa}}^\kappa(y)+
\phi_\kappa(y_1\dots y_n)^{1-s}\lambda_{-m_{u,\kappa}}^\kappa(y).
\end{align*}
Later we will see that $m_{u,\kappa}$ can be assumed to be the same
vector $m_u$ for every cusp, at least if it can be defined.
If however $\eta_\kappa=\phi_\kappa=0$ holds for all $\kappa\in\mathcal{S}$
(i.e. $u$ is a cusp form), 
then we simply set $m_{u}=0$.
Aside from the next paragraph, 
in the following we always assume that $0< \ree s<1$ holds
whenever the number $s$ is associated with the form $u$.

Later we will make use of a specific family automorphic forms, namely
the Eisenstein series that are defined as follows. 
Let $\kappa\in\mathcal{S}$ be a cusp. For an $s\in\C$ and $m\in\Z^{n-1}$ 
the Eisenstein series belonging to $\kappa$ is given by
\begin{equation}\label{Eisenstein_def}
E_\kappa(z,s,m)=
\sum_{\gamma\in\Gamma_\kappa\setminus \Gamma}
y(\gamma^{(1)}z_1)^{s_1}\dots y(\gamma^{(n)}z_n)^{s_n}
\end{equation}
for any $z\in\HH^n$, where the exponents $s_1,\dots,s_n$ are defined as in (\ref{s_k_def}). 
This series
converges absolutely and uniformly on compact subsets for $\ree s>1$.
Also, $E_\kappa(z,s,m)$ (as a function in the variable $z$)
is clearly a $\Gamma$-invariant
eigenfunction of the Laplacians and 
(as a function of $s$) it can be continued 
meromorphically  to the whole complex plane.
Moreover, for a cusp $\kappa'\in\mathcal{S}$ 
the coefficient $\eta_{\kappa'}$ in the  Fourier coefficient $a_{\kappa'}(y)$
is $1$ if $\kappa'=\kappa$ and $0$ otherwise. For the details see chapter II and 
also section III.4 of \cite{efrat}.

\section{The geometric trace}

\subsection{The automorphic kernel}
In the following we 
 fix  a compactly supported smooth function 
$\psi\in C^\infty_0(\R^n)$ and define the \emph{point pair invariant kernel}
\[
k_\psi(z,w)=
k(z,w)=\psi\left(\frac{\abs{z_1-w_1}^2}{\im z_1\cdot \im w_1},
\dots,\frac{\abs{z_n-w_n}^2}{\im z_n\cdot \im w_n}\right)
=\psi\left(\frac{\abs{z-w}^2}{\im z\cdot \im w}\right)
\]
for every $z,w\in\HH^n$. 
Invariance means that 
$k(z,w)=k(\sigma z,\sigma w)$ holds for every $z,w\in\HH^n$ and $\sigma\in \psl$.
To avoid long formulae we often use
the latter compact notation for $\psi$ and its transforms defined below.
In these cases the operations on vectors always indicate coordinate-wise operations.
The \emph{automorphic kernel} $K(z,w)$ is given by the sum
\begin{equation}\label{kernel_K_def}
K(z,w) = \sum_{\gamma\in\Gamma}k(z,\gamma w)
\end{equation}
that clearly defines an automorphic function w.r.t. $\Gamma$.

The following
transformations of $\psi$ 
 often occur in computations:
\begin{align}\label{2_1_Q_g_h_def}
Q(w)=Q(w_1,\dots,w_n) &
:= \int_{w_n}^\infty\dots\int_{w_1}^\infty
\frac{\psi(t_1,\dots,t_n)}{\sqrt{t_1-w_1}\dots\sqrt{t_n-w_n}}
\,dt_1\dots dt_n,\nonumber
\\[3mm]
g(u)=g(u_1,\dots,u_n)&
:=Q(e^{u_1}+e^{-u_1}-2,\dots,e^{u_n}+e^{-u_n}-2),
\\[3mm]
h(r)=h(r_1,\dots, r_n)&
:=
\int_{-\infty}^{\infty}
\dots
\int_{-\infty}^{\infty}
g(u_1,\dots, u_n)e^{i\sum_{k=1}^nr_ku_k}
\,du_1\dots du_n.\nonumber
\end{align}
Note that this is the multidimensional version of the Harish-Chandra transform. 
Since $\psi$ is a compactly supported smooth function,
$g$ is also a smooth function with compact support and hence $h$ is rapidly decreasing.

The inverses of the transforms above are 
\begingroup
\allowdisplaybreaks
\begin{align}\label{2_1_transform_rules}
g(u_1,\dots,u_n)&
=
\frac{1}{(2\pi)^n}
\int_{-\infty}^{\infty}
\dots
\int_{-\infty}^{\infty}
h(r_1,\dots,r_n)e^{-i\sum_{k=1}^nr_ku_k}
\,dr_1\dots dr_n,
\nonumber
\\[9mm]
Q(w_1,\dots,w_n)&
=g\left(2\log\left(\sqrt{\frac{w_1}{4}+1}+\sqrt{\frac{w_1}{4}}\right),\dots,
        2\log\left(\sqrt{\frac{w_n}{4}+1}+\sqrt{\frac{w_n}{4}}\right)\right),
\\[6mm]
\psi(t_1,\dots, t_n)
 &
= \frac{(-1)^n}{\pi^n}\int_{t_n}^\infty\dots\int_{t_1}^\infty
\frac{\frac{\partial^n Q}{\partial w_1\dots\partial w_n}(w_1,\dots,w_n)}{\sqrt{w_1-t_1}\dots\sqrt{w_n-t_n}}
\,dw_1\dots dw_n\nonumber
\end{align}
\endgroup
(see Proposition I.2.2 in \cite{efrat}).

\subsection{The geometric trace}
Now we turn to the multidimensional version of the generalized Selberg trace formula,
more precisely, to the geometric trace that is computed by collecting the terms
of the conjugacy classes in the sum (\ref{kernel_K_def}). As in \cite{biro1},
our starting point is the integral
\[
\Tr_u K = \int_F K(z,z)u(z)\,\mu(z),
\]
where $u$ is a fixed automorphic form that satisfies 
the growth condition $u(z)=o(e^{2\pi y_k})$ for $k=1,\dots,n$,
$F$ is the fundamental domain of $\Gamma$, and 
$\mu$ is the product measure on $\HH^n$ obtained from the measure
$y^{-2}\,dx\,dy$ on $\HH$. 
Note that $\ree s_k<1$ is assumed for each $k$ (excluding the case
$u(z)=1$, that would yield the trace formula given in \cite{efrat}).
Since this integral is not necessarily convergent, we work with the \emph{truncated trace}
defined by
\begin{equation}\label{trace_def}
\Tr_u^AK := \int_{F_A} K(z,z)u(z)\,d\mu(z)
\end{equation}
for every $A>0$, where 
\[
F_A=F_0\cup\left(\bigcup_{\kappa\in\mathcal{S}}F_\kappa^A\right)
\]
with 
$F_\kappa^A=\{z\in F_\kappa:\,Y_0^\kappa(z)\leq A\}$.

Substituting the definition of $K(z,w)$ into (\ref{trace_def}) and summing 
over the conjugacy
classes in $\Gamma$ we get
\[
\Tr_u^A K= \sum_{\{\gamma\}}\sum_{\sigma\in\{\gamma\}}\int_{F_A}k(z,\sigma z)u(z)\,d\mu(z),
\]
where $\{\gamma\}$ denotes the conjugacy class of an element $\gamma\in\Gamma$.
Note that the conjugacy class of the identity element consists only of itself,
and the term that belongs to it is a constant multiple of the integral
\[
\int_{F_A}u(z)\,d\mu(z).
\]
This integral converges as $A\to \infty$ and the limit is zero since
the Laplacians are symmetric operators and the eigenvalues of $1$ and $u$ are different.

Our aim is to give the contribution of the different types of classes in this trace.
The main result can be summarized in the form
\[
\Tr_u^A K=\Sigma_{\mathrm{ell}}+\Sigma_{\mathrm{mix}}+\Sigma_{\mathrm{par}}
+\Sigma_{\mathrm{hyp-par}},
\]
where the four terms on the right hand side stand for the contribution of
totally elliptic, mixed (and totally hyperbolic), totally parabolic and hyperbolic-parabolic classes,
respectively. Note that the totally hyperbolic classes can be handled 
in the same way as the mixed classes, hence they are melted in a single term above.
Since the individual terms are given by lengthy and complicated formulae, we 
do not give the whole sum in one statement, but split 
the main result into four theorems below instead. We begin with
the contribution of elliptic, mixed and totally hyperbolic classes,
here the corresponding
results are similar to the ones in \cite{biro1}.
Our main focus is therefore on the 
parabolic and hyperbolic-parabolic classes, that are handled afterwards.

\subsection{Contribution of totally  elliptic, mixed and totally hyperbolic classes}\label{ell_mix_sec}

In these cases the sum 
\[
 \sum_{\sigma\in\{\gamma\}}\int_{F_A}k(z,\sigma z)u(z)\,d\mu(z)
\]
in $\Tr_u^A K$ belonging to the class $\{\gamma\}$ actually converges as $A\to\infty$ (one can see this by analysing the detailed
computations in the proofs of the next two theorems). We will also see that there are only finitely many
classes for which the sum above is non-zero, hence we can integrate over $F$ instead of
$F_A$ (by including an $o(A)$ term as well).

Since $\sigma_1^{-1}\gamma\sigma_1 = \sigma_2^{-1}\gamma\sigma_2$ holds if and only if $\sigma_2\sigma_1^{-1}$ is in the centralizer $C(\gamma)$ of $\gamma$, and
this is equivalent to $\sigma_2\in C(\gamma)\sigma_1$, we get that
\begin{equation}\label{tgamma_def}
T_\gamma:=\sum_{\sigma\in\{\gamma\}}\int_{F}k(z,\sigma z)u(z)\,d\mu(z)
=
\sum_{\sigma\in C(\gamma) \setminus\Gamma}\int_{F}
k(z,\sigma^{-1}\gamma\sigma z)u(z)\,d\mu(z).
\end{equation}
As $k(\varrho z,\varrho w) = k(z,w)$ holds for every $\varrho\in \psl$ and $u$ is invariant under the
action of $\Gamma$, this last sum is
\[
\sum_{\sigma\in C(\gamma)\setminus \Gamma}
\int_{F}
k(\sigma z,\gamma\sigma z)u(\sigma z)\,d\mu(z) 
=
\int_{C(\gamma)\setminus \HH^n}
k(z,\gamma z)u(z)\,d\mu(z),
\]
and for every $\varrho\in \psl$  this can be written as
\begin{equation}\label{tgamma_final_form}
\int_{\varrho^{-1}(C(\gamma)\setminus \HH^n)}
k(\varrho z,\gamma \varrho z)u(\varrho z)\,d\mu(z)=
\int_{(\varrho^{-1}C(\gamma)\varrho)\setminus \HH^n}
k(z, \varrho^{-1}\gamma \varrho z)u(\varrho z)\,d\mu(z)
\end{equation}
since the measure $\mu$ and the function $k$ are $\psl$ invariant.
Note that 
$(\varrho^{-1}C(\gamma)\varrho)\setminus \HH^n$
is nothing else but the fundamental domain of the group $\varrho^{-1}C(\gamma)\varrho$.

Now we turn to the contribution of totally elliptic classes.
Let us first note that
by Corollary $2.14_1$ in \cite{freitag} there 
are only finitely many such classes, hence
$\Sigma_\mathrm{ell}$ is a sum of finitely many terms of the form (\ref{tgamma_final_form}).

Before giving the value $\Sigma_\mathrm{ell}$ let us fix the following notations.
Every elliptic element $\gamma\in\Gamma$ is conjugate in $PSL(2,\R)^n$ to an element
of the form
$(R(\theta_\gamma^{(1)}),\dots,R(\theta_\gamma^{(n)}))$, where
\begin{equation}\label{elll_R_def}
R(\alpha)=\mtx{\cos\alpha}{\sin\alpha}{-\sin\alpha}{\cos\alpha}
\end{equation}
and the vector $(\theta_\gamma^{(1)},\dots,\theta_\gamma^{(n)})\in[0,\pi)^n$ depends only on the 
conjugacy class of $\gamma$.

Besides, let $g_{\lambda}(r):[0,\infty)\to\C$ be the unique solution of the differential equation
\begin{equation}\label{elliptic_diff_eq}
g''(r)+\frac{\cosh r}{\sinh r}g'(r)=\lambda g(r)
\end{equation}
satisfying the initial condition $g(0)=1$.

\begin{thm}\label{elliptic_main_thm}
The contribution
of the totally elliptic classes in the truncated trace, i.e., the value
of $\Sigma_{\mathrm{ell}}$ is
\[
\sum_{\{\gamma\}\textrm{ t.ell.}
}
\frac{(2\pi)^n}{m_\gamma}u(z_\gamma)
\int\limits_0^\infty\dots\int\limits_0^\infty
\psi(S(r_1,\theta_\gamma^{(1)}),\dots,S(r_n,\theta_\gamma^{(n)}))
\left(\prod_{k=1}^ng_{\lambda_k}(r_k)\sinh r_k\,dr_k\right)+o(A),
\]
where the sum runs over all totally elliptic classes and for every class $\{\gamma\}$
the point $z_\gamma\in\HH^n$ is the fixed point of $\gamma$, $m_\gamma\in\N^+$ is
the order of the centralizer of $\gamma$, $S(r,\theta)=(2\sinh r\sin\theta)^2$
for any $r,\tht\in\R$, 
$(\lambda_1,\dots,\lambda_n)$ is the Laplacian eigenvalue vector of $u$, 
and the functions $g_{\lambda_k}$ and the vector $(\theta_\gamma^{(1)},\dots,\theta_\gamma^{(n)})$
are defined above the theorem. Moreover, 
there are only finitely many totally elliptic conjugacy classes, hence the sum above is finite.
\end{thm}

Next we handle the mixed and totally hyperbolic classes, i.e. the classes whose 
elements have at least one hyperbolic coordinates. For simplicity, we assume that 
the first $1\leq m \leq n$ coordinates of the element are hyperbolic, while
the following $n-m$ coordinates are elliptic. The results below can easily 
be reformulated and proved for other distributions of coordinates of
different types. Note also that the case $m=n$ is
the totally hyperbolic case, 
at least if the fixed points
of the element are not cusps, which is assumed in this section. 

Any mixed or totally hyperbolic element  $\gamma$ 
is conjugate in $PSL(2,\R)^n$ to an element of the form
\begin{equation}\label{mixed_conj}
\nu=(D(N_\gamma^{(1)}),\dots,D(N_\gamma^{(m)}),R(\theta_\gamma^{(m+1)}),\dots,R(\theta_\gamma^{(n)}))
\end{equation}
for some $N_\gamma^{(k)}>1$ and $\theta_\gamma^{(l)}\in[0,\pi)$, where
\begin{equation}\label{mix_D_def}
D(N)=\mtx{N^{1/2}}{0}{0}{N^{-1/2}}
\end{equation}
and $R(\theta)$ is defined in (\ref{elll_R_def}) above.
It is not hard to see that all these numbers
are determined uniquely by the class of $\gamma$ (and hence the notations $N_\gamma^{(k)}$ and
$\theta_\gamma^{(l)}$ are justified). 
The number $N_\gamma^{(k)}$ is called the  \emph{norm}
of $\gamma^{(k)}$.
We also set
\begin{equation}\label{mixed_N_def}
 N(\vartheta,\gamma^{(k)}):= \frac{N_\gamma^{(k)}+(N_\gamma^{(k)})^{-1}-2}{\cos^2\tht}
\end{equation}
for any $\vartheta\in (-\frac{\pi}{2};\frac{\pi}{2})$ and $1\leq k\leq m$.

Let $\varrho_\gamma\in \psl$ be an element  such that
$\nu =\varrho_\gamma^{-1}\gamma\varrho_\gamma$ holds.
To give the contribution of the class $\{\gamma\}$ 
we need to describe that centralizer $C(\nu)$ 
of $\nu$ in $\varrho_\gamma^{-1}\Gamma\varrho_\gamma$.
By 
the results of section I.5 in
\cite{efrat} 
the centralizer $C(\gamma)$  of $\gamma$ is a free abelian group of rank $m$.
We fix a set of its generators denoted by $\gamma_1,\dots,\gamma_m$,
then the centralizer $C(\nu)\leq\varrho_\gamma^{-1}\Gamma\varrho_\gamma$ is
$\varrho_\gamma^{-1}C(\gamma)\varrho_\gamma$ 
and it is generated by the elements $\nu_i =\varrho_\gamma^{-1}\gamma_i\varrho_\gamma$ 
for $i=1,\dots,m$. 
As the $\gamma_i$'s 
have the same fixed points as $\gamma$ this is true also for the conjugates and therefore
\[
\nu_i =\left(
    D(N_{\gamma_i}^{(1)}),\dots,D(N_{\gamma_i}^{(m)}),
    R(\theta_{\gamma_i}^{(m+1)}),\dots,R(\theta_{\gamma_i}^{(n)})
\right).
\]
 This is a somewhat imprecise notation since 
$N_{\gamma_i}^{(k)}>1$ may not be assured for all $k$. This means that $N_{\gamma_i}^{(k)}$ is
not necessarily the norm of $\nu_i^{(k)}$ in the above sense, but it is
still determined by the (fixed) generator $\nu_i$
and we keep using this notation.
The action of the first $m$ coordinates of the elements 
$\nu_i$ is simple:
for every $z=(z_1,\dots,z_n)\in\HH^n$  and $k=1,\dots,m$ we have
$|\nu_i^{(k)}z_k|=N_{\gamma_i}^{(k)}\abs{z_k}$ and $\arg \nu_i^{(k)} z_k = \arg z_k$.
The next statement follows now easily:
\begin{prop}\label{totally_hyp_fundom_prop}
The fundamental domain 
of the centralizer $C(\nu)=\varrho_\gamma^{-1}C(\gamma)\varrho_\gamma$ is 
\[
F_{C(\nu)}=\{z\in\HH^n:\,(\log\abs{z_1},\dots,\log\abs{z_m})\in P_\gamma\}
\]
where
$P_\gamma$ is the fundamental parallelepiped of the lattice
in $\R^m$ generated by the vectors
\[
(\log N_{\gamma_i}^{(1)},\dots,\log N_{\gamma_i}^{(m)})\qquad(i=1,\dots,m).
\]
\end{prop}

Before the next theorem we introduce one more notation. Let $f_{\lambda}(\tht)$ be the unique solution of the differential equation
\begin{equation}\label{mixed_diff_eq}
F''(\tht) = \frac{\lambda}{\cos^2\tht}F(\tht)\qquad(\tht\in(-\pi/2,\pi/2))
\end{equation}
with the initial condition $f_{\lambda}(0) = 1$
and $f'_{\lambda}(0) = 0$. We are now ready to state

 \begin{thm}\label{mixed_main_thm}
 The contribution
of the mixed and totally hyperbolic classes in the truncated trace
is
\[
\Sigma_{\mathrm{mix}}=\sum_{\{\gamma\}\textrm{ mixed or totally hyperbolic}}T_\gamma
+o(A),
\]
where for a class $\{\gamma\}$, for which the first $m$ coordinates of $\gamma$ are hyperbolic
and the rest are elliptic, the value of $T_\gamma$ is
\begin{align*}
&(2\pi)^{n-m}F_\gamma(0,\dots,0)\times\\
&\quad\times\int\limits_0^\infty\dots\int\limits_0^\infty
\int\limits_{-\frac{\pi}{2}}^{\frac{\pi}{2}}
\dots
\int\limits_{-\frac{\pi}{2}}^{\frac{\pi}{2}}
\psi(N(\vartheta_1,\gamma^{(1)}),\dots,
N(\vartheta_m,\gamma^{(m)}),
S(r_{m+1},\theta_\gamma^{(m+1)}),\dots,S(r_n,\theta_\gamma^{(n)}))\\
&\qquad\qquad\qquad\qquad\qquad\times
\left(\prod_{k=1}^mf_{\lambda_k}(\tht_k)\,\frac{d\tht_k}{\cos^2\tht_k}\right)
\left(\prod_{k=m+1}^ng_{\lambda_k}(r_k)\sinh r_k)\,dr_{k}\right),
\end{align*}
where $N(\vartheta,\gamma^{(k)})$ is defined in (\ref{mixed_N_def}),
$S(r,\theta)=(2\sinh r\sin\theta)^2$
for any $r,\tht\in\R$, 
$(\lambda_1,\dots,\lambda_n)$ is the Laplacian eigenvalue vector of $u$, 
the functions $f_{\lambda_k}$ were defined before the theorem, 
the functions $g_{\lambda_k}$ were defined before
Theorem \ref{elliptic_main_thm}, the vector
$(\theta_\gamma^{(m+1)},\dots,\theta_\gamma^{(n)})$ is given by (\ref{mixed_conj})
and
\[
F_\gamma (0,\dots,0) = 
\int\limits_{(\log r_1,\dots,\log r_m)\in P_\gamma}u(\varrho_\gamma^{(1)}(r_1i),\dots, \varrho_\gamma^{(m)}(r_mi),
\varrho_\gamma^{(m+1)}i,\dots,\varrho_\gamma^{(n)}i)
\prod_{k=1}^m\frac{dr_k}{r_k}.
\]     
Here $\varrho_\gamma\in\psl$ is an element for which $\varrho_\gamma^{-1}\gamma\varrho$ is
of the form (\ref{mixed_conj}) and the set $P_\gamma$ is given in Proposition 
\ref{totally_hyp_fundom_prop}. An analogous formula gives the value of $T_\gamma$ when
the $m$ hyperbolic and $n-m$ elliptic coordinates are distributed differently.
Moreover, the value of $T_\gamma$ is zero except for finitely many classes.
\end{thm}

\subsection{Contribution of hyperbolic-parabolic classes}\label{hyp-par_sec}

We continue with the contribution of those classes whose elements
have only hyperbolic coordinates but also fix a cusp. 
Let $\gamma=(\gamma^{(1)},\dots,\gamma^{(n)})$ be  such an element and
let $x=(x_1,\dots,x_n)$ be a cusp fixed by $\gamma$. 
That is, $x_i$ is a fixed point of the hyperbolic coordinate $\gamma^{(i)}$ and we denote
its other one by $x_i'$.
Then, by the results of \S20 in \cite{shimizu} the fixed point
$(x_1',\dots,x_n')$ of $\gamma$ is also a cusp.

Every hyperbolic-parabolic class is represented 
by an element that fixes a  cusp $\kappa\in\mathcal{S}$.
An  element of this type is conjugated by the scaling element 
$\sigma_\kappa\in \psl$ to an element of the form
\[
\gamma_{m,\alpha}^\kappa:= 
\mtx{(u^\kappa_m)^{1/2}}{\alpha (u^\kappa_m)^{-1/2}}{0}{(u^\kappa_m)^{-1/2}},
\]
where $\alpha\in\tk_\kappa$, $m=(m_1,\dots,m_{n-1})\in \Z^{n-1}\setminus\{0\}$ and 
$u^\kappa_m=(\eps_1^\kappa)^{m_1}\dots(\eps_{n-1}^\kappa)^{m_{n-1}}\in\Lambda_\kappa$. 
The cusp in the notation $u_m^\kappa$ indicates that this unit depends also on the
multiplier group $\Lambda_\kappa$ and its generators. However, as a byproduct of 
the proof of this section's main result we also get the following 
\begin{prop}\label{multi_group_prop}
    The multiplier group $\Lambda_\kappa$ is the same for any $\kappa\in\mathcal{S}$.
\end{prop}
This fact allows us to drop the index from the notation of
the multiplier group and we simply write $\Lambda$ in the following.
Also, we can and will fix
 the same generators $\eps_1,\dots,\eps_{n-1}$ for every cusp and 
 therefore it is legitimate to write $u_m$ instead of $u_m^\kappa$.
It follows also that the matrix $\mathcal{E}_\kappa$ and consequently
the integer vector $m_{u,\kappa}$ (defined in (\ref{eps_s_eqs})) are 
independent of $\kappa$ and will be denoted simply by $\mathcal{E}$ and $m_u$,
respectively.
 Note that the lattice $\tk_\kappa$ does depend on the cusp $\kappa$.
 
The element $\gamma_{m,\alpha}^\kappa$ 
fixes the points $\infty$ and $q= \frac{\alpha}{1-u_m}$
and according to the first paragraph of this section both points are cusps for
$\sigma_\kappa^{-1}\Gamma\sigma_\kappa$.
We will denote by $\tilde{\kappa}_{m,\alpha}\in\mathcal{S}$ the cusp for $\Gamma$ that can 
be taken (by an element of $\Gamma$) to $\sigma_\kappa q$.

The centralizer $C(\gamma_{m,\alpha}^\kappa)$ of $\gamma_{m,\alpha}^\kappa$ in
$\sigma_\kappa^{-1}\Gamma\sigma_\kappa$ is given  in \S20 of \cite{shimizu}:
\begin{prop}\label{hyp_par_centralizer}
The centralizer $C(\gamma_{m,\alpha}^\kappa)$ of the element $\gamma_{m,\alpha}^\kappa$ 
is a free abelian group of rank $n-1$ generated by
some elements $\gamma(l_1),\dots,\gamma(l_{n-1})$, where 
$l_j\in \Z^{n-1}\setminus\{0\}$ for any $1\leq j \leq n-1$ and
\[
\gamma(l_j)=\mtxskip{u_{l_j}^{1/2}}{ \frac{ u_{l_j}-1 }{u_m-1}\alpha u_{l_j}^{-1/2} }{0}{ u_{l_j}^{-1/2}}.
\]
\end{prop}

In the following we fix a generating set of elements described in the proposition above and 
 define 
the $(n-1)\times(n-1)$ matrix 
\begin{equation}\label{L_mtx_def}
L_{m,\alpha}^{\kappa}:=\left[
\begin{array}{cccc}
l_1^{(1)}  &l_2^{(1)}  & \dots & 
l_{n-1}^{(1)}\\
l_1^{(2)}  & l_2^{(2)}  &\dots & 
l_{n-1}^{(2)}\\
\vdots &\vdots  & \ddots & \vdots\\
 l_1^{(n-1)} &l_2^{(n-1)} & \dots 
& l_{n-1}^{(n-1)} 
\end{array}
\right].
\end{equation}
As before, 
we need to describe  the fundamental domain 
$F_{C(\gamma_{m,\alpha}^\kappa)}$ of $C(\gamma_{m,\alpha}^\kappa)$.
One shows by induction that 
$\gamma(l_j)^h =\gamma(hl_j)$ holds for any $h\in\Z$. 
Let $C$
denote the group generated by the (clearly independent) elements
\begin{align}\label{rho_l_def}
\rho_{l_j}:=\mtx{u_{l_j}^{1/2}}{0}{0}{u_{l_j}^{-1/2}}\qquad(1\leq j \leq n-1).
\end{align}
We 
set
$T =\mtx{1}{-\frac{\alpha}{1-u_{m}}}{0}{1}$,
then 
$C(\gamma_{\alpha,m}^\kappa) = T^{-1} C T$
and hence if $F_{C}$ is a fundamental domain for $C$, then
\[
F_{C(\gamma_{m,\alpha}^\kappa)}=T^{-1} F_{C} = F_{C}+\frac{\alpha}{1-u_{m}}=F_{C}+q
\]
is a fundamental domain for $C(\gamma_{\alpha,m}^\kappa)$.

It remains to describe the fundamental domain $F_C$.
As in the case of mixed and totally hyperbolic elements we use polar coordinates. That is, 
for a point $z=(z_1,\dots,z_n)\in\HH^n$ we write 
$z_k=r_ke^{i(\pi/2+\tht_k)}$ where $r_k\in \R^+$ and
$-\frac{\pi}{2} < \tht_k < \frac{\pi}{2}$ ($1\leq k\leq n$). 
Let $\tilde{P}_{m,\alpha}^\kappa$ be the fundamental domain of the $n-1$ dimensional lattice in
$V=\{a\in\R^n:\,a_{1}+\dots+a_{n}=0\}$
generated by the vectors 
$v_j=l_j^{(1)}\log\eps_1+\dots +l_j^{(n-1)}\log \eps_{n-1}$ 
($1\leq j \leq n-1$). For later purposes we specify the choice of 
$\tilde{P}_{m,\alpha}^\kappa$, namely we take the shifted image of the parallelpiped 
spanned by $v_1,\dots,v_{n-1}$ (in $V$) symmetric to the origin.
Let us fix the unit vector $\mathbf{1}=(n^{-\frac{1}2},\dots,n^{-\frac{1}2})^T$, 
it spans the subspace $V^\perp$.
If $P_{m,\alpha}^\kappa=\{t \mathbf{1}+\tilde{P}_{m, \alpha}^\kappa:\, t\in\R\}$,
then the fundamental domain
$F_{C}$ is given by
\[
(\tht_1,\dots,\tht_n)\in(-\pi/2;\pi/2),\qquad 
(\log r_1,\dots, \log r_n)\in P_{m,\alpha}^\kappa.
\]

The contribution of the class belonging to $\gamma_{m,\alpha}^\kappa$ 
in the truncated trace can be divided into
two parts. A main term (that diverges as $A\to\infty$) comes from the zeroth Fourier
coefficient of $u(\sigma_\kappa (z+q))$ and the transformed zeroth coefficient of $u(\sigma_{\tilde{\kappa}_{m,\alpha}}(z+q))$
while we obtain the remaining convergent part by subtracting these  from 
$u(\sigma_\kappa (z+q))$. Note that here the argument is shifted 
since we will give the result in terms of the fundamental domain $F_C$.
For any cusp $\kappa'\in\mathcal{S}$ we set
\[
M_{\kappa'}(z):=
\eta_{\kappa'} y_1^{s_1}\dots y_n^{s_n}
+\phi_{\kappa'} y_1^{1-s_1}\dots y_n^{1-s_n},
\]
 and subtract $M_\kappa(z+q)=M_\kappa(z)$ from $u(\sigma_\kappa (z+q))$, while
 in the case of
 $\tilde{\kappa}_{m,\alpha}$ we first apply a transformation that maps $q$ to $\infty$.
This is performed by an element 
$\sigma_{\tilde{\kappa}_{m,\alpha}}^{-1}\gamma\sigma_\kappa\in\sigma_{\tilde{\kappa}_{m,\alpha}}^{-1}\Gamma\sigma_\kappa$
 having the matrix form
$\mtx{e}{f}{\frac{u_m-1}{\delta}}{\frac{\alpha}{\delta}}$
where $\delta=e\alpha+f(1-u_m)$. Using this notation
 we define
\[
\tilde{u}_{m,\alpha}^\kappa(z):=u\left(\sigma_\kappa(z+q)\right)
-M_\kappa\left(z\right)
-M_{\tilde{\kappa}_{m,\alpha}}\left(-\frac{\delta^2E_m^{-2}u_m^{-1}}{z}\right)
\]
where $E_m:=u_m^{-1/2}-u_m^{1/2}$. Note that it is convenient to work with the quantity 
$E_m$ because of its skew-symmetry in $m$. We mention in advance
that though the vector $\delta$ depends on the choice of 
$\sigma_{\tilde{\kappa}_{m,\alpha}}^{-1}\gamma\sigma_\kappa$, but the 
norm of $\delta^2$ depends only on $m$ and $\alpha$.
Note also that  
the translation invariance of $M_{\tilde{\kappa}_{m,\alpha}}(z)$ was also used
to simplify the defining formula of $\tilde{u}_{m,\alpha}$.

Before the main statement of the section we define an
equivalence relation
on the lattice $\tk_\kappa$ for any $\kappa$: 
the elements $\alpha,\beta\in \tk_\kappa$ are said to be equivalent
if
$\beta  = (u_{m}-1)
a+u_{l}\alpha$
holds
 for some $l\in\Z^{n-1}$ and $a\in\tk_\kappa$, that is, if and only if
 $\beta$ and  $u_l\alpha$ represent
the same element in the finite factor group 
$\tk_\kappa^m:=\tk_\kappa/(u_m-1)\tk_\kappa$. 
These classes (represented as elements of $\tk_\kappa^m/\Lambda$) are used
to list the hyperbolic-parabolic conjugacy classes in the next result:
\begin{thm}\label{hyp-par_thm}
 The contribution
of the hyperbolic-parabolic
classes in the truncated trace is
\[
\Sigma_{\mathrm{hyp-par}}
=\delta_{m_{u}}M(A)+
\sum_{\kappa\in\mathcal{S}}\,\,\sum_{m\in\Z^{n-1}\setminus \{0\}}
\,\,\sum_{\alpha\in \tk_\kappa^m/\Lambda}
C_\kappa(m,\alpha)+o(A),
\]
the main term $M(A)$ is given by
\begin{equation}\label{hyp_par_main_term}
\frac{\abs{\det \mathcal{E}}}{n}
\sum_{\kappa\in\mathcal{S}}
\left(
\frac{\eta_\kappa A^s}{s}
+\frac{\phi_\kappa A^{1-s}}{1-s}
\right)
\sum_{m\in\Z^{n-1}\setminus \{0\}}
g(\log u_m)
\end{equation}
where $\mathcal{E}$ and $g$ were defined in (\ref{eps_s_eqs}) and 
(\ref{2_1_Q_g_h_def}), respectively, 
and the term $C_\kappa(m,\alpha)$ is
\begin{align*}
\frac{1}{2}\int\limits_{\log r\in P_{m,\alpha}^\kappa}\tilde{u}_{m,\alpha}^\kappa(ri)
\prod_{k=1}^n\frac{dr_k}{r_k}&
\int\limits_{-\frac{\pi}{2}}^{\frac{\pi}{2}}
\dots
\int\limits_{-\frac{\pi}{2}}^{\frac{\pi}{2}}
\psi\left(
    \frac{E_m^2}{\cos^2\tht}
\right)
\left(\prod_{k=1}^n\frac{f_{\lambda_k}(\tht_k)\,d\tht_k}{\cos^2\tht_k}\right),
\end{align*}
where $(\lambda_1,\dots,\lambda_n)$ is the Laplacian eigenvalue vector of $u$ and
 $f_{\lambda_k}(\tht)$ is the unique solution of the differential equation
 (\ref{mixed_diff_eq})
 satisfying the initial condition $f_{\lambda_k}(0) = 1$ and $f'_{\lambda_k}(0) = 0$.
Moreover, the terms $C_\kappa(m,\alpha)$ and the terms in (\ref{hyp_par_main_term})
are zero for any cusp $\kappa$  for all but finitely many $m$.
\end{thm}

\subsection{The \texorpdfstring{$\zeta$}{zeta}-functions at the cusps
}\label{zeta_section}

In this section we introduce the $\zeta$-function
belonging to the lattice $\tk_\kappa$ and the multiplier group $\Lambda$ where
$\kappa$ is any cusp in $\mathcal{S}$, these will be needed for the last part of our result 
in the next section.
In fact we define these objects in the following  general situation.
Let $L\leq \R^n$ be a 
lattice of full rank for which
the following holds: if $l=(l^{(1)},\dots,l^{(n)})^T\in L$ and $l^{(k)}=0$
for any $1\leq k\leq n$, then $l=0$. 
We define the
\emph{norm} of $l$ by $N l=\prod_{k=1}^nl^{(k)}$. Using this terminology,
we assume that any non-zero element of $L$ has non-zero norm. 
In addition, let $M\leq (\R^+)^n$ be a 
discrete norm-$1$ multiplicative subgroup of rank $n-1$ 
that acts on $\R^n$ by coordinate-wise multiplication so that $L$ is invariant under this action. 
That is, let us assume that $M\cong \Z^{n-1}$,
for every $\eps\in M$ $N\eps =1$ holds, and finally,
for any $\eps\in M$ and $l\in L$ we have 
$\eps l:=(\eps^{(1)}l^{(1)},\dots,\eps^{(n)}l^{(n)})^T\in L$.
We remark 
that $N(\eps l)=N l$ holds for every $\eps\in M$ and $l \in \R^n$
and hence the norm, restricted to the lattice $L$, is in fact defined on $M$-orbits.

Though it is more convenient to give some required technical statements in the above
context,
it is important to mention that this generality is illusory. Namely, it 
only simplifies the notations
and helps to focus on the important properties of the underlying objects, but does not
give a wider point of view since the setting 
we talk about is basically the same
as in our initial  situation above.
More concretely, a slight modification of the proof of 
Theorem 4 in \cite{shimizu}  gives the following
\begin{prop}\label{zeta_prop}
Assume that $L$ and $M$ are as above. Then there exists a totally real number 
field
$\Q\leq K$ of degree $n$ with embeddings $K^{(1)}\subset \R,\dots,K^{(n)}\subset \R$
such that for each $\eps=(\eps^{(1)},\dots,\eps^{(n)})^T\in M$ the coordinate
$\eps^{(k)}$ is a totally positive unit in $K^{(k)}$ for all $1\leq k\leq n$ and the coordinates of 
$\eps$ are conjugates of each other.
Moreover, there is
a vector $\nu=(\nu_1,\dots,\nu_n)^T\in\R^n$ with non-zero coordinates such that 
for any element $\alpha\in \nu\cdot L$ (where the coordinate-wise product is taken) 
we have $\alpha^{(k)}\in K^{(k)}$ for every $k$ and
the coordinates of $\alpha$ are conjugates of each other.
\end{prop}

Let us fix the generators $\eps_1,\dots,\eps_{n-1}$ of the group $M$ and define
the matrices $\mathcal{E}_M$ and $\mathcal{E}_M^{-1}$ analogously as 
$\mathcal{E}(=\mathcal{E}_\kappa)$ and its 
inverse were  defined 
in Section \ref{Fourier_expansion_section}. Then 
the corresponding
Grössencharacter-type exponential sum $\lambda_{M,m}$ is given for every $m\in\Z^{n-1}$ 
in the same way as $\lambda^\kappa_m$ was in
(\ref{1_3_grossen_exp_sum}).

We are interested in the sum
\begin{align}\label{zeta_def}
Z_{L,M}(s,m):=\sum_{0\neq l\in L/M}\frac{\lambda_{M,-m}(\abs{l})}{\abs{Nl}^s},
\end{align}
where $s\in\C$, $m\in\Z^{n-1}$ and $\abs{l}$ is the vector whose coordinates are the absolute values 
of the corresponding coordinates of $l$. 
Since $\lambda_{M,-m}$ is a
multiplicative function and it is trivial on $M$, the function $Z_{L,M}(s,m)$ 
is well-defined 
and it will be called the zeta function belonging to the lattice $L$ and
the  group $M$.
For $L=\tk_\kappa$
and $M=\Lambda$ we simply write $Z_\kappa(s,m)$ and it will be called 
the \emph{zeta function of $\Gamma$ belonging to the cusp} $\kappa$.
Note that the index of $\lambda$ in (\ref{zeta_def}) is $-m$ 
in order to obtain the equivalent
form
\[
Z_{L,M}(s,m)=
\sum_{0\neq l\in L/M}\frac{1}{\abs{l_1}^{s_1}\dots \abs{l_n}^{s_n}},
\]
where $s_1,\dots,s_n$ are defined  as in (\ref{s_k_def}). 

We will show a few properties of these functions, they are summarized in
the following lemma:
\begin{lemma}\label{zeta_lemma}
The sum in (\ref{zeta_def}) converges absolutely and locally uniformly  for $\ree s>1$, hence
it defines an analytic function on this half-plane. It can be continued meromorphically
to the whole plane $\C$ and
 has no poles on $\C$ except for the case $m=0$ when $s=1$ and $s=0$ are the only 
 poles
 of $Z_{L,M}(s,0)$, they are simple and 
 \[
 \mathrm{Res}_{s=1}Z_{L,M}(s,0)= \frac{2^n\abs{\det \mathcal{E}_M}}{n\cdot\mathrm{vol}(\R^n/L) }.
 \] 
 Moreover, the completed function
\[
\Xi_{L,M}(s,m):=
\pi^{-\frac{ns}{2}}\left(\prod_{k=1}^n\Gamma\left(\frac{s_k}{2}\right)\right)Z_{L,M}(s,m)
\]
satisfies the functional equation
\[
\mathrm{vol}(\R^n/L)^{1/2}\Xi_{L,M}(s,m)
=\mathrm{vol}(\R^n/L^*)^{1/2}\Xi_{L^*,M}(1-s,-m)
\]
where $L^*$ is the dual lattice of $L$.
The convexity bound
\[
Z_{L,M}(s,m)\ll_{\eps,m}\abs{\im s}^{n(1-\ree s)/2+\eps}
\]
holds for any $\eps>0$ if $0\leq \ree s\leq 1$ and $\abs{\im s}$ is bounded from below by some positive constant.
Also, $Z_{L,M}(s,m)\ll_{\eps,m}\abs{\im s}^\eps$ holds if $\ree s>1$ and $\abs{\im s}$ is
bounded away from zero.
\end{lemma}

\subsection{Contribution of totally parabolic classes}

At last we  give the contribution of the totally parabolic classes in the geometric
trace. In advance of that we introduce some notations.
If $\kappa$ is a cusp of $\Gamma$ and $\sigma_\kappa$ is 
the corresponding scaling element, then the 
zeroth Fourier coefficient of
$u(\sigma_\kappa z)$ is
$a_\kappa(y)=\eta_\kappa y_1^{s_1}\dots y_n^{s_n}+
\phi_\kappa y_1^{1-s_1}\dots y_n^{1-s_n}$.
Recall that if at least one of $\eta_\kappa$ and $\phi_\kappa$ is non-zero
(i.e. when $u$ does not vanish at $\kappa$)  then 
the zeroth coefficient 
$a_\kappa(y)$
can be written in the form
$\eta_\kappa (y_1\dots y_n)^s\lambda_{m_{u}}(y)+
\phi_\kappa(y_1\dots y_n)^{1-s}\lambda_{-m_{u}}(y)$, where 
$s=\frac{s_1+\dots +s_n}{n}$ and
$m_{u}\in\Z^{n-1}$. If $\eta_\kappa=\phi_\kappa=0$ for all $\kappa$, then $m_u=0$
holds by definition. Note that in (\ref{1_3_grossen_exp_sum}) 
the function $\lambda_{m}=\lambda^\kappa_{m}$ was 
defined  in terms of the entries of $\mathcal{E}_\kappa=\mathcal{E}$ and hence by Proposition
\ref{multi_group_prop} and its subsequent paragraph $\lambda_m$ is independent of the cusp that (from now on) will not be included in
our notation.
Now we are in the position to give $\Sigma_\mathrm{par}$ explicitly:
\begin{thm}\label{parabolic_thm}
The contribution
of the totally parabolic classes in the truncated trace, i.e., the value
of $\Sigma_{\mathrm{par}}$ is
\begin{align*}
&\sum_{\kappa\in\mathcal{S}}
 \delta_{m_{u}}\frac{\abs{\det \mathcal{E}}}{n}
 \left(
 \frac{\eta_\kappa A^{s}}{s}
+
 \frac{\phi_\kappa A^{1-s}}{1-s}
\right)g\left(0\right)
\\[1mm]
&\qquad\qquad\qquad\qquad\qquad
+\mathrm{vol}(\R^n/\tk_\kappa)(\eta_\kappa Z_\kappa(1-s,-m_{u})F(0)+
\phi_\kappa Z_\kappa(s,m_{u})\tilde{F}(0))
+o(1)
 \end{align*}
as $A\to\infty$, 
where
\begin{align*}
F(S)=
\int\limits_{(\R^+)^n}
\psi\left(t^2\right)
\prod_{k=1}^nt_k^{S-s_k}
\,dt_k\quad\textrm{and}\quad
\tilde{F}(S)=
\int\limits_{(\R^+)^n}
\psi\left(t^2\right)
\prod_{k=1}^n
t_k^{S+s_k-1}
\,dt_k.
\end{align*}
The values $F(0)$ and $\tilde{F}(0)$ can be expressed in terms of the function
$h$ defined in (\ref{2_1_Q_g_h_def}):
\[
F(0) = 
\left(\frac{i}{2^{2-s}\pi^2}\right)^n
\left(
    \prod_{k=1}^n
        \Gamma\left(\frac{1-s_k}{2}\right)^2
\right)
\int\limits_{(\R^+)^n}
h(r)\prod_{k=1}^n 
\frac{\Gamma(\frac{s_k}{2}+ir_k)}{\Gamma(\frac{2-s_k}{2}+ir_k)}
r_k\,dr_k,
\]
and
\[
\tilde{F}(0) = 
\left(\frac{i}{2^{s+1}\pi^2}\right)^n
\left(
    \prod_{k=1}^n
        \Gamma\left(\frac{s_k}{2}\right)^2
\right)
\int\limits_{(\R^+)^n}
h(r)\prod_{k=1}^n 
\frac{\Gamma(\frac{1-s_k}{2}+ir_k)}{\Gamma(\frac{s_k+1}{2}+ir_k)}
r_k\,dr_k.
\]
\end{thm}

\section{Proofs of the theorems}

\subsection{Proof in the totally elliptic case}

We first prove  Theorem \ref{elliptic_main_thm}.
As it was already mentioned, there are only finitely many elliptic conjugacy classes by Corollary $2.14_1$ in \cite{freitag}),
hence it remains to show the formula for $T_\gamma$ (defined in (\ref{tgamma_def})). 
We prove by induction.
Since our argument is very similar to the one in \cite{biro1},  we only sketch the induction step.

Let $\gamma\in\Gamma$ be a totally elliptic element with the elliptic fixed point $z_\gamma\in \HH^n$.
The centralizer $C(\gamma)$ consists of the elements in $\Gamma$ which leave  the point $z_\gamma$ fixed
(see \cite{shimizu}, p. 37) and the stabilizer $\Gamma_{z_\gamma}$ of $z_\gamma$ in $\Gamma$ is a finite cyclic group
(see Remark 2.14 in \cite{freitag}).
Let us denote by $m_\gamma$ the order of $C(\gamma)$.
Every elliptic element in $PSL(2,\R)$ is conjugate to an element of the form  
$\mtx{\cos\theta}{\sin\theta}{-\sin \theta}{\cos\theta}$, hence the generator $\gamma_0$ of $C(\gamma)$ can be chosen so that it is conjugate in $PSL(2,\R)^n$ to the element
\[
\gamma'_0 = 
\left(	
\mtxx{\cos\ds{\frac{\pi}{m_\gamma}}}{\sin\ds{\frac{\pi}{m_\gamma}}}{-\sin\ds{\frac{\pi}{m_\gamma}}}{\cos\ds{\frac{\pi}{m_\gamma}}   }{3},
\mtxx{\cos\ds{\frac{l_2\pi}{m_\gamma}}}{\sin\ds{\frac{l_2\pi}{m_\gamma}}}{-\sin\ds{\frac{l_2\pi}{m_\gamma}}}{\cos\ds{\frac{l_2\pi}{m_\gamma}}}{3},
\dots,
\mtxx{\cos\ds{\frac{l_n\pi}{m_\gamma}}}{\sin\ds{\frac{l_n\pi}{m_\gamma}}}{-\sin\ds{\frac{l_n\pi}{m_\gamma}}}{\cos\ds{\frac{l_n\pi}{m_\gamma}}}{3}
\right),
\]
where $l_k\in\Z$ with $\gcd(l_k,m_\gamma)=1$ for every $k=2,\dots,n$.
Let us write $\gamma'=\varrho^{-1}\gamma\varrho$,
we give the fundamental domain $F_{C(\gamma')}$ of $C(\gamma')=\left<\gamma'_0\right>\leq \varrho^{-1}\Gamma\varrho$.
The first coordinate of $\gamma'_0$ 
is a rotation around the point $i\in\HH$ by the angle $2\pi/m_\gamma$, therefore 
every $C(\gamma')$-orbit has
exactly one point in the set
$F_0\times \HH^{n-1}$, where $F_0\subset \HH$ is a sector enclosed by two half-lines with endpoint $i$ 
and angle $2\pi/m_\gamma$. Note that in fact each coordinate is a rotation around $i$ which means that $\varrho$
takes the point $(i,\dots,i)$ to the fixed point $z_\gamma$ of $\gamma$. 
Now by (\ref{tgamma_final_form}) we have
\[
T_\gamma=\int_{F_{C(\gamma')}}k(z,\gamma' z)u(\varrho z)\,d\mu(z)=\frac{1}{m_\gamma}\int_{\HH^n}k(z,\gamma' z)u(\varrho z)\,d\mu(z),
\]
where we used the $PSL(2,\R)^n$-invariance of the function $k$ and the measure $\mu$, the $\Gamma$-invariance of $u$ and that $\gamma'$ and
$\gamma'_0$ commute. Writing $z=(z_1,\dots,z_n)$ we have
\begin{equation}\label{ellitptic_multiple_integral}
\int_{\HH^n}k(z,\gamma' z)u(\varrho z)\,d\mu(z)=
\int_\HH\dots\int_\HH k(z,\gamma'z)u(\varrho z)\,d\mu(z_1)\dots\, d\mu(z_n),
\end{equation}
where $\mu(z_k)$ denotes the measure $y_k^{-2}dx_k\,dy_k$.
In the inner integral above the coordinates $z_2,\dots,z_n$ are fixed,
and the function $u(\varrho z)$ can be regarded as
a function of $z_1$. It is the eigenfunction of
the Laplace operator $\Delta_1$ (because the operator commutes with the group action), furthermore, 
the value of $k(z,\gamma' z)$ depends only on the hyperbolic distance of $z_1$ and $\gamma'^{(1)}z_1$.
To simplify the notation we 
write $u(\varrho z) = u_{1}(z_1)$ and 
$k(z,w) = k_{1}(z_1,w_1)$.
Furthermore, as $\gamma'$ is fixed we can simply write
$k(z,\gamma' z) = k_{1}(z_1,\gamma'^{(1)}z_1)$.
With this notation the inner integral becomes
\[
T_{1}:=\int_\HH k(z,\gamma'z)u(\varrho z)\,d\mu(z_1)=\int_{\HH}k_{1}(z_1,\gamma'^{(1)}z_1)u_{1}(z_1)\,d\mu(z_1).
\]

Recall that in (\ref{elll_R_def}) we introduced the notation
\[
R(\alpha)=\mtx{\cos\alpha}{\sin\alpha}{-\sin\alpha}{\cos\alpha},
\]
for any $\alpha\in\R$.
For a vector
$\varphi=(\varphi^{(1)},\dots,\varphi^{(n)})\in\R^n$ we set 
$R(\varphi)=(R(\varphi^{(1)}),\dots, R(\varphi^{(n)}))$. 
The elements of the centralizer $C(\gamma')$ are of the form $R(\varphi)$, in particular
$\gamma'=R(\theta_\gamma)$ for some $\theta_\gamma=(\theta_\gamma^{(1)},\dots,\theta_\gamma^{(n)})$ 
where $\theta_\gamma^{(k)}\in[0,\pi)$.
As it was mentioned in Section \ref{ell_mix_sec} it is not hard to see that
the vector $\theta_\gamma$ is determined by 
the class of $\gamma$, i.e. it is independent of
the choice of $\varrho$ (at least if every coordinate is chosen from the interval $[0,\pi)$).
Since $\gamma'\in C(\gamma')$ we have in fact $\theta_\gamma^{(k)}=l_k\pi /m_\gamma$ for some integer 
$0<l_k<m_\gamma$ ($k=1,\dots,n$).

Next we use geodesic polar coordinates (see \cite{iwaniec}, section 1.3), i.e. we make
the substitution 
$z_1=R(\varphi_1)e^{-r_1}i$ 
where $r_1\in(0,\infty)$ is the hyperbolic distance of $i$ and $z_1$ and $\varphi_1\in[0,\pi)$. 
Then we have
$d\mu(z_1) = (2\sinh r_1)\,dr_1\,d\varphi_1$
and 
\[
T_{1}= 
\int_0^\infty\int_0^\pi k_{1}(R(\varphi_1)e^{-r_1}i,R(\theta_\gamma^{(1)})R(\varphi_1)e^{-r_1}i))
u_{1}(R(\varphi_1)e^{-r_1}i)2\sinh r_1\,d\varphi_1\,dr_1.
\]
As the elements $R(\theta_\gamma^{(1)})$ and $R(\varphi_1)$ commute and $k_{1}$ depends only on the hyperbolic distance of 
the variables, we get that
\[
T_{1}=\int_0^\infty k_{1}(e^{-r_1}i,R(\theta_\gamma^{(1)})e^{-r_1}i))
\left(\int_0^\pi  u_{1}(R(\varphi_1)e^{-r_1}i)\,d\varphi_1\right)(2\sinh  r_1)\,dr_1.
\]

Recall that
$k_{1}(z_1,w_1)=\psi\left(\rho(z_1,w_1), \dots,\rho(z_n,w_n)\right)=:
\psi_{1}(\rho(z_1,w_1))$,
where 
\[
\rho(z_k,w_k)=\frac{\abs{z_k-w_k}^2}{\im z_k\, \im w_k}
\]
for $k=1,\dots,n$.
One gets by a computation that
\begin{align*}
\rho(e^{-r_k}i,R(\theta_\gamma^{(k)})e^{-r_k}i)=
\frac{\abs{-e^{-2r_k}+1}^2\sin^2\theta_\gamma^{(k)}}{e^{-2r_k}}=
(2\sinh r_k\sin\theta_\gamma^{(k)})^2,
\end{align*}
hence
\begin{equation}\label{T1}
T_{1}=\int_0^\infty \psi_{1}((2\sinh r_1\sin\theta_\gamma^{(1)})^2)
\left(\int_0^\pi  u_{1}(R(\varphi_1)e^{-r_1}i)\,d\varphi_1\right)(2\sinh  r_1)\,dr_1.
\end{equation}
Let us define the function 
\[
G_{1}(w)=\frac{1}{\pi}\int_0^\pi u_{1}(R(\varphi_1)w)\,d\varphi_1,
\]
where $w\in\HH$.
By Lemma 1.10 in \cite{iwaniec} the value of $G_{1}$ depends only on the hyperbolic distance $r$ of  $w$ and $i$.
Moreover, $G_{1}$ is the eigenfunction of the (one dimensional) Laplace operator $\Delta$ with eigenvalue $\lambda_1$, where
$\lambda_1$ is the first coordinate of the eigenvalue vector of $u$. 
Now by Lemma 1.12 of \cite{iwaniec} this function is unique up to a constant factor.
Furthermore
\[
\Delta=
\frac{\partial^2}{\partial r^2}+
\frac{\cosh r}{\sinh r}\frac{\partial}{\partial r}+
\frac{1}{4\sinh^2 r}\frac{\partial^2}{\partial \varphi^2}
\]
hence the function $G_{1}$ (as a function of $r$) satisfies the differential equation
(\ref{elliptic_diff_eq}) with the constant $\lambda=\lambda_1$, and consequently
\[
G_{1}(w)
=g_{\lambda_1}(r)u_{1}(i)
=g_{\lambda_1}(r)u(\varrho^{(1)} i,\varrho^{(2)} z_2,\dots,\varrho^{(n)}z_n),
\]
where $g_{\lambda_1}(r):[0,\infty)\to\C$ is the solution of (\ref{elliptic_diff_eq}) 
with $\lambda=\lambda_1$
satisfying the initial condition $g_{\lambda_1}(0)=1$.
By substituting this in (\ref{T1}), then interchanging the integrals in 
(\ref{ellitptic_multiple_integral}) and proceeding by induction one gets the statement 
of the theorem.

\subsection{Proof in the mixed and totally hyperbolic cases}

We continue with the proof of Theorem \ref{mixed_main_thm}. 
Let $\gamma\in\Gamma$ a mixed or a totally hyperbolic element. 
We assume that the first $1\leq m\leq n$ coordinate
of $\gamma$ are hyperbolic while the following $n-m$ coordinates are elliptic, the proof of
the statement is similar in the other cases.
We have seen in Section \ref{ell_mix_sec} that such an element is conjugated 
by an element $\varrho\in PSL(2,\R)^n$ to an element of the form
$\nu =\left(
    D(N_\gamma^{(1)}),\dots,D(N_\gamma^{(m)}),R(\theta_\gamma^{(m+1)}),\dots,R(\theta_\gamma^{(n)})
\right)$
for some $N_\gamma^{(k)}>1$ and $\theta_\gamma^{(l)}\in[0;2\pi)$,
where $D(N)$ and  $R(\theta)$ were defined in (\ref{mix_D_def}) and (\ref{elll_R_def}), respectively.
An easy computation shows
that these numbers are uniquely defined by the class
$\{\gamma\}$.

As in the totally elliptic case, by (\ref{tgamma_final_form}) we need to consider the 
integral
\[
T_\gamma = \int_{F_{C(\nu)}}k(z,\nu z)u(\varrho z)\,d\mu(z),
\]
where $F_{C(\nu)}$ is the fundamental domain for the centralizer $C(\nu)\leq \varrho^{-1}\Gamma\varrho$ of 
$\nu=\varrho^{-1}\gamma\varrho$.
This domain was described in Proposition \ref{totally_hyp_fundom_prop} 
and the notations introduced there will be used in the following.

For the first $m$ coordinates of $z$  we change to polar coordinates, i.e. make the substitution $z_k=r_ke^{i(\pi/2+\tht_k)}$ where $r_k\in(0,\infty)$ 
and $\tht_k\in(-\frac{\pi}{2},\frac{\pi}{2})$ ($k=1,\dots, m$),
while for the last $n-m$ coordinates we change to geodesic polar coordinates
as in the previous proof.
A simple computation shows that $\rho(z_k,\nu^{(k)} z_k) =N(\vartheta_k,\gamma^{(k)})$
for $k=1,\dots,m$ (where $N(\vartheta_k,\gamma^{(k)})$ was defined in (\ref{mixed_N_def})),
and by
$y_k^{-2}\,dx_k\,dy_k = (r_k\cos^2\tht_k)^{-1}\,dr_k\,d\tht_k$ and
the results of the previous proof we
obtain that $T_\gamma$ is
\begin{align*}
(2\pi)^{n-m}
&\int\limits_0^\infty\dots\int\limits_0^\infty
\int\limits_{-\frac{\pi}{2}}^{\frac{\pi}{2}}
\dots
\int\limits_{-\frac{\pi}{2}}^{\frac{\pi}{2}}
\psi(N(\vartheta_1,\gamma^{(1)}),\dots,
N(\vartheta_m,\gamma^{(m)}),
S(r_{m+1},\theta_{\gamma}^{(m+1)}),\dots,S(r_n,\theta_\gamma^{(n)}))\\
&\qquad\,\,\times
F(e^{i(\frac{\pi}{2}+\tht_1)},\dots,e^{i(\frac{\pi}{2}+\tht_m)})
\,\frac{d\tht_1}{\cos^2\tht_1}\dots \frac{d\tht_m}{\cos^2\tht_m}
\left(\prod_{k=m+1}^ng_{\lambda_k}(r_k)\sinh r_k\,dr_k\right),
\end{align*}
where $S(r,\theta)=(2\sinh r\sin\theta)^2$ and
\[
F(z)=\int\limits_{(\log r_1,\dots,\log r_m)\in P_\gamma}
u(\varrho^{(1)} (r_1z_1),\dots,\varrho^{(m)} (r_mz_m),
\varrho^{(m+1)}i,\dots,\varrho^{(n)}i)
\prod_{k=1}^m\frac{dr_k}{r_k}
\]
for any $z\in\HH^m$. Since $u(\varrho z)$ is invariant under the action of the centralizer 
$C(\nu)$, one
sees easily that the function $F$ is invariant under is coordinate-wise scalar multiplication, i.e. 
$F(R_1z_1,\dots,R_mz_m)=F(z_1,\dots,z_m)$ holds
for any $R_1, \dots,R_m\in(0,\infty)$ and 
$z_1,\dots,z_m\in\HH$.
This means that
$F$  depends only on the vector 
$(\vartheta_1,\dots,\vartheta_m)$ (where $z_k=r_ke^{i(\frac{\pi}{2}+ \vartheta_k)}$).

Moreover, since $u$ is the eigenfunction of every $\Delta_k$ with eigenvalue $\lambda_k$  
and these operators commute with the group
action, we infer that $F(z)$ is also an eigenfunction of the Laplacians $\Delta_1,\dots,\Delta_m$ with the same corresponding eigenvalues. 
As
\[
\Delta_k = (r_k\cos\tht_k)^2
\left(
\frac{\partial^2}{\partial r_k^2}
+
r_k^{-1}\frac{\partial}{\partial r_k}
+
r_k^{-2}\frac{\partial^2}{\partial\tht_k^2}
\right),
\]
we obtain the differential equations
\begin{equation}\label{diff_eq_0}
\frac{\partial^2 F}{\partial\tht_k^2}(\tht_1,\dots,\tht_m) = \frac{\lambda_k}{\cos^2\tht_k}F(\tht_1,\dots,\tht_m)\qquad
(\tht_k\in(-\pi/2,\pi/2),\,k=1,\dots,m).
\end{equation}
Let $f_{\lambda_k}(\tht)$ be the unique solution of the differential equation (\ref{mixed_diff_eq})
with $\lambda = \lambda_k$ and 
the initial conditions $f_{\lambda_k}(0) = 1$ and $f'_{\lambda_k}(0) = 0$, and
$\tilde{f}_{\lambda_k}(\tht)$ the one with $\tilde{f}_{\lambda_k}(0) =0$ and 
$\tilde{f}'_{\lambda_k}(0) =1$. Note that $f_{\lambda_k}(-\tht)$ satisfies (\ref{mixed_diff_eq})
and the initial conditions of $f_{\lambda_k}(\tht)$ and hence they agree, i.e. $f_{\lambda_k}$
is an even function. Similarly, $\tilde{f}_{\lambda_k}$ is an odd function.

The equation  
\[
F(\tht_1,\dots,\tht_m) = F(0,\tht_2,\dots,\tht_m)f_{\lambda_1}(\tht_1)+\frac{\partial F}{\partial \tht_1}(0,\tht_2,\dots,\tht_m)\tilde{f}_{\lambda_1}(\tht_1)
\]
holds by (\ref{diff_eq_0}) 
for every fixed $\tht_2,\dots,\tht_m$, 
hence the inner integral in $T_\gamma$ is
\begingroup
\allowdisplaybreaks
\begin{align*}
&\int\limits_{-\frac{\pi}{2}}^{\frac{\pi}{2}}
\psi(N(\vartheta_1,\gamma^{(1)}),\dots,
N(\vartheta_m,\gamma^{(m)}),
S(r_{m+1},\theta_\gamma^{(m+1)}),\dots,S(r_n,\theta_\gamma^{(n)}))\\ &\qquad\qquad\qquad\qquad\qquad\times
\left(
    F(0,\tht_2,\dots,\tht_m)f_{\lambda_1}(\tht_1)
    +
    \frac{\partial F}{\partial \tht_1}(0,\tht_2,\dots,\tht_m)
    \tilde{f}_{\lambda_1}(\tht_1)
\right)
\frac{d\tht_1}{\cos^2\tht_1}\\
&=\int\limits_{-\frac{\pi}{2}}^{\frac{\pi}{2}}
\psi(N(\vartheta_1,\gamma^{(1)}),\dots,
N(\vartheta_m,\gamma^{(m)}),
S(r_{m+1},\theta_\gamma^{(m+1)}),\dots,S(r_n,\theta_\gamma^{(n)}))\\ 
&\qquad\qquad\qquad\qquad
\qquad\qquad\qquad\qquad\qquad\qquad\qquad\times
    F(0,\tht_2,\dots,\tht_m)f_{\lambda_1}(\tht_1)
    \frac{d\tht_1}{\cos^2\tht_1}
\end{align*}
\endgroup
because $N(\tht_1,\gamma^{(1)})$ and $\cos^{-2}(\tht_1)$ are even and  $\tilde{f}_{\lambda_1}(\tht_1)$ is
an odd function. Then, by induction 
(using also that $f_{\lambda_k}$ is even) we infer
\begin{align*}
T_\gamma=&(2\pi)^{n-m}F(0,\dots,0)\times\\
&\times\int\limits_0^\infty\dots\int\limits_0^\infty
\int\limits_{-\frac{\pi}{2}}^{\frac{\pi}{2}}
\dots
\int\limits_{-\frac{\pi}{2}}^{\frac{\pi}{2}}
\psi(N(\vartheta_1,\gamma^{(1)}),\dots,
N(\vartheta_m,\gamma^{(m)}),
S(r_{m+1},\theta_\gamma^{(m+1)}),\dots,S(r_n,\theta_\gamma^{(n)}))\\
&\qquad\qquad\qquad\qquad\qquad\times
\left(\prod_{k=1}^mf_{\lambda_k}(\tht_k)\,\frac{d\tht_k}{\cos^2\tht_k}\right)
\left(\prod_{k=m+1}^ng_{\lambda_k}(r_k)\sinh r_k)\,dr_{k}\right),
\end{align*}
where
\[
F(0,\dots,0) = 
\int\limits_{(\log r_1,\dots,\log r_m)\in P_\gamma}u(\varrho^{(1)}(r_1i),\dots, \varrho^{(2)}(r_mi),
\varrho^{(m+1)}i,\dots,\varrho^{(n)}i)
\prod_{k=1}^m\frac{dr_k}{r_k}.
\]

It remains to show that there are only finitely many mixed or
totally hyperbolic equivalence classes for which 
$T_\gamma$  is non-zero.
Note that since $\psi$ has compact support and
\[
 N(\vartheta_k,\gamma^{(k)})=
 \frac{N_\gamma^{(k)}+(N_\gamma^{(k)})^{-1}-2}{\cos^2\tht_k}\geq 
 N_\gamma^{(k)}+(N_\gamma^{(k)})^{-1}-2 = \abs{\tr [\gamma^{(k)}]}-2,
\]
 we get $T_\gamma = 0$ once $\abs{\tr[\gamma^{(k)}]}$ is big enough for some $k$.
Hence it is enough to show that there are only finitely many classes whose representatives
have
hyperbolic coordinates of bounded  norm (and trace).

By Theorem \ref{1_3_commensurable_thm} $\Gamma$ is commensurable with a Hilbert modular group $\Gamma_K$,
so there is an $M_\gamma\in\N^+$ such that $\gamma^{M_\gamma}$ is conjugate to an element in 
$\Gamma_K$, moreover,
the  exponent $M_\gamma$ is bounded by a constant depending on $\Gamma$.
It follows that it is enough to show that there are only finitely many 
mixed or totally hyperbolic classes in $\Gamma_K$
with coordinates of bounded trace.

Note that the trace of $(\gamma^{M_\gamma})^{(k)}$  is in $\OO_{K^{(k)}}$ for every $1\leq k\leq n$,
and these values are conjugates of each other. Hence if each of them is bounded,
then the norm of them is bounded as well, so there are finitely many possibilities for the values
of these traces.
Finally, by Proposition I.7.1 and the paragraph after
Definition I.7.2 in  \cite{efrat}, there are only finitely many totally hyperbolic conjugacy classes with
fixed traces, and this completes the proof of Theorem \ref{mixed_main_thm}.

\subsection{Proof in the hyperbolic-parabolic case}

In this section we prove Theorem \ref{hyp-par_thm}. In the course of
the following proof the statement of Proposition \ref{multi_group_prop} will also be 
verified, but until that point of our argument we  always indicate any possible 
or evident  dependence on a cusp. Accordingly, we temporarily use the notations $\Lambda_\kappa$ for
the multiplier group and $\eps_j^\kappa$ for its generators.

Recall that every hyperbolic-parabolic class is represented 
by an element that fixes a cusp $\kappa\in\mathcal{S}$
and such an element is conjugated by the scaling element 
$\sigma_\kappa\in \psl$ to an element of the form 
\begin{equation}\label{hyp-par_conj_form}
\gamma_{m,\alpha}^\kappa= 
\mtx{(u_m^\kappa)^{1/2}}{\alpha (u_{m}^\kappa)^{-1/2}}{0}{(u_{m}^\kappa)^{-1/2}},
\end{equation}
where $\alpha\in\tk_\kappa$, $m=(m_1,\dots,m_{n-1})\in \Z^{n-1}\setminus\{0\}$ and 
$u_m^\kappa=(\eps_1^\kappa)^{m_1}\dots(\eps_{n-1}^\kappa)^{m_{n-1}}\in\Lambda_\kappa$. 
To simplify the notation, we often write $u_m$ instead of $u_m^\kappa$ in the following,
at least when $\kappa$ is fixed.
Note that the action of
$\gamma_{m,\alpha}^\kappa$ on a point $z\in\HH^n$ can be written as
$\gamma_{m,\alpha}^\kappa z=u_m^\kappa z+\alpha$ and hence 
$q= \frac{\alpha}{1-u_m^\kappa}$ is the real fixed vector of $\gamma_{m,\alpha}^\kappa$
that is also a cusp for $\sigma_\kappa^{-1}\Gamma\sigma_\kappa$ (as it was already 
mentioned in Section \ref{hyp-par_sec}).

A simple computation shows that if 
 elements $\gamma_{m,\alpha}$ and $\gamma_{{m'},\beta}$ of the form 
 (\ref{hyp-par_conj_form}) are conjugate
in $\sigma_\kappa^{-1}\Gamma\sigma_\kappa$, then
 $m'=m$ or
$m' = -m$ holds.
Assume first that the elements $\gamma_{m,\alpha}$ and $\gamma_{m,\beta}$ are
conjugate to each other. Again, it follows by a straightforward calculation
that 
\[
\beta  = (u_{m}^\kappa-1)
a+(u_{l}^\kappa)^{-1}\alpha
\]
holds for some $l\in\Z^{n-1}$ and $a\in\tk_\kappa$ in this case.
This means exactly that $\beta$ represents
the same element in the finite factor group 
$\tk_\kappa^m:=\tk_\kappa/(u_m^\kappa-1)\tk_\kappa$
as $(u_l^\kappa)^{-1}\alpha$, and hence for a fixed $m$ (and $\kappa$) the 
hyperbolic-parabolic classes
are represented by the  equivalence classes of $\tk_\kappa^m/\Lambda_\kappa$. 

Now assume that 
the elements
$\gamma_{m,\alpha}$ and $\gamma_{-m,\beta}$ are conjugate for some 
$m\in\Z^{n-1}\setminus \{0\}$ and $\alpha,\beta\in\tk_\kappa$, i.e. 
$\tau^{-1} \gamma_{m,\alpha}\tau =\gamma_{-m,\beta}$ for
some $\tau\in\sigma_\kappa^{-1}\Gamma\sigma_\kappa$. 
Since $\tau^{-1} \gamma_{m,\alpha}\tau$ 
fixes 
$\tau^{-1}\infty$ and also $\tau^{-1} q$, one of these points must be $\infty$.
If $\tau^{-1}\infty=\infty$ was true, then 
the conjugate would be of the form $\gamma_{m,\beta}$, which is impossible
(since $m\neq0$).
It follows that $\tau^{-1}$ takes  $q$  to $\infty$, hence these cusps are equivalent in 
$\sigma_\kappa^{-1}\Gamma\sigma_\kappa$. 
Similarly, if these cusps are equivalent then $\gamma_{m,\alpha}$
 is conjugate to an element $\gamma_{-m,\beta}$.

Based on this the contribution of the hyperbolic-parabolic classes in the trace
can be written
as
\begin{align}\label{hyp_par_total_sum}
&\frac{1}{2}\sum_{\kappa\in\mathcal{S}}\,
\sum_{m\in\Z^{n-1}\setminus 0}\,\,
\sum_{\alpha\in \tk_\kappa^m/\Lambda_\kappa}
\sum_{
\begin{array}{c}
\scriptstyle{
\sigma\in C(\gamma) \setminus\Gamma
}\\
\scriptstyle{
\gamma\sim\gamma_{m,\alpha}^\kappa
}
\end{array}
}
\int_{F_A}
k(z,\sigma^{-1}\gamma\sigma z)u(z)\,d\mu(z)=\nonumber\\
&\quad=\frac{1}{2}
\sum_{\kappa\in\mathcal{S}}\,
\sum_{m\in\Z^{n-1}\setminus 0}\,\,
\sum_{\alpha\in \tk_\kappa^m/\Lambda_\kappa}\,\,
\sum_{\sigma\in C(\gamma_{m,\alpha}^\kappa)\setminus 
\sigma_\kappa^{-1}\Gamma\sigma_\kappa}
\int_{\sigma (\sigma_\kappa^{-1} F_A) }k(z,\gamma_{m,\alpha}^\kappa z)
u(\sigma_\kappa z)\,d\mu(z),
\end{align}
where $C(\gamma_{m,\alpha}^\kappa)$ is the 
centralizer of $\gamma_{m,\alpha}^\kappa$ in $\sigma_\kappa^{-1}\Gamma\sigma_\kappa$.
We multiply the whole sum by $1/2$ since every class is taken into account
for both fixed cusps of their elements except for those whose fixed points
are equivalent in $\Gamma$. But in the latter case  we count these classes twice
for an $m$ and $-m$ as well.

Let us focus on the inner sum
\begin{equation}\label{hpint}
\sum_{\sigma\in C(\gamma_{m,\alpha}^\kappa)\setminus \sigma_\kappa^{-1}\Gamma\sigma_\kappa}
\int_{\sigma (\sigma_\kappa^{-1} F_A) }k(z,\gamma_{m,\alpha}^\kappa z)
u(\sigma_\kappa z)\,d\mu(z).
\end{equation}
First, note that $k(z,\gamma_{m,\alpha}^\kappa z)$ can be written as
\begin{align*}
\psi
\left(
\frac{(E_m^{(1)}x_1-\alpha_1(u_m^{(1)})^{-1/2} )^2}{y_1^2} + 
(E_m^{(1)})^2,\dots,
\frac{(E_m^{(n)}x_n-\alpha_n(u_m^{(n)})^{-1/2} )^2}{y_n^2 } + (E_m^{(n)})^2
\right),
\end{align*}
where $E_m=E_m^\kappa=u_m^{-1/2}-u_m^{1/2}$. Since $\psi$ is compactly supported,
it follows immediately that (\ref{hpint}) is zero for all but finitely many $m$, and
hence the sum in (\ref{hyp_par_total_sum}) is finite.

The union of the sets $\sigma(\sigma_\kappa^{-1} F_A)$ in the integrals above,
where  $\sigma$ runs through the right cosets of the centralizer $C(\gamma_{m,\alpha}^\kappa)$,
makes up the fundamental domain $F_{C(\gamma_{m,\alpha}^\kappa)}$ of
$C(\gamma_{m,\alpha}^\kappa)$ except for the
images of the part $F\setminus  F_A =:F^*_A$. We will now show that for some
cosets it is unnecessary to omit the images of $F_A^*$, since there
we integrate only the zero function.
For this we write the kernel function in the form
\begin{equation}\label{kernel_at_cusps}
k(z,\gamma_{m,\alpha}^\kappa z)=\psi
\left(
\frac{\abs{z_1-q_1 }^2}{(E_m^{(1)})^{-2}y_1^2},\dots,
\frac{\abs{z_n-q_n}^2}{(E_m^{(n)})^{-2}y_n^2}
\right).
\end{equation}
The part 
$\sigma (\sigma_\kappa^{-1} F^*_A)$ is the same as 
\[
\bigcup_{\kappa'}\{z\in F_{C({\gamma_{m,\alpha}^\kappa})}:\,\sigma_\kappa\sigma^{-1}z\in F,\, Y_0^{\kappa'}(\sigma_\kappa\sigma^{-1}z)\geq A \}.
\]
The condition $Y_0^{\kappa'}(\sigma_\kappa\sigma^{-1}z)\geq A$
means that $z\in \sigma\sigma_\kappa^{-1}\sigma_{\kappa'}U_A$, that is,
there is a $w\in U_A$ such that $ z=\sigma\sigma_\kappa^{-1}\sigma_{\kappa'}w$.
Let us define $\nu=\sigma\sigma_\kappa^{-1}\sigma_{\kappa'}=\mtx{a}{b}{c}{d}$. If $w=u+iv$, then
the product of the expressions in the arguments on the right hand side of 
(\ref{kernel_at_cusps}) is
\begin{align*}
\prod_{k=1}^n\frac{\abs{\nu_k w_k-q_k }^2}{(E_m^{(k)})^{-2}y_k(\nu w)^2}
&=
\frac{N(E_m^2)}{Y_0(\nu w)^2}
\prod_{k=1}^n
\abs{\nu_{k} w_k-\frac{\alpha_k}{1-u_m^{(k)}}}^2\\[3mm]
&=
\frac{N((1-u_m)^2)}{Y_0(\nu w)^2}
\prod_{k=1}^n
\abs{\frac{a_kw_k+b_k}{c_kw_k+d_k}-\frac{\alpha_k}{1-u_m^{(k)}}}^2\\[3mm]
&\geq
\frac{N((1-u_m)^2)}{Y_0(\nu w)^2}
\prod_{k=1}^n
\frac{((1-u_m^{(k)})a_k-\alpha_k c_k)^2v_k^2}{(1-u_m^{(k)})^2\abs{c_kw_k+d_k}^2}
\\[3mm]
&
=N((1-u_m)a-\alpha c)^2\prod_{k=1}^n\abs{c_kw_k+d_k}^2\\[3mm]
&\geq N((1-u_m)a-\alpha c)^2N(c)^2 Y_0(w)^2.
\end{align*}

The first two factors of the last product are bounded away from zero by Lemma $2.9_1$ in
\cite{freitag}, at least if they are non-zero. This follows for the second factor easily
but requires some explanation in the case of the first factor.
The point $q$ is a cusp for $\sigma_\kappa^{-1}\Gamma\sigma_\kappa$,
hence $q=\sigma_\kappa^{-1}\beta$ for some cusp $\beta$ of $\Gamma$. 
Let $\lambda\in\mathcal{S}$ denote the base element 
of the $\Gamma$-equivalence class of $\beta$, then
$q=\sigma_\kappa^{-1}\gamma_0\lambda$ for some $\gamma_0\in\Gamma$,
i.e. $\sigma_\lambda^{-1}\gamma_0^{-1}\sigma_\kappa\in\sigma_\lambda^{-1}\Gamma\sigma_\kappa$ 
takes $q$ to $\infty$, hence it is of the form 
$\mtx{e}{f}{\frac{u_m-1}{\delta}}{\frac{\alpha}{\delta}}$,
where $\delta=e\alpha+f(1-u_m)$. 

We first show  that $N(\delta^2)$ depends only on $q$, i.e. on $m$ and $\alpha$.
Assume that $\sigma_\lambda^{-1}\gamma'\sigma_\kappa$
also takes $q$ to $\infty$ for some $\gamma'\in\Gamma$, and therefore it
is of the form
$\mtx{e'}{f'}{\frac{u_m-1}{\delta'}}{\frac{\alpha}{\delta'}}$. Now
\[
(\sigma_\lambda^{-1}\gamma'\sigma_\kappa)(\sigma_\lambda^{-1}\gamma_0^{-1}\sigma_\kappa)^{-1}
=\sigma_\lambda^{-1}\gamma'\gamma_0\sigma_\lambda\in\sigma_\lambda^{-1}\Gamma\sigma_\lambda
\]
fixes $\infty$ (which is a cusp for $\sigma_\lambda^{-1}\Gamma\sigma_\lambda$), and then
\[
\mtx{e'}{f'}{\frac{u_m-1}{\delta'}}{\frac{\alpha}{\delta'}}
\mtx{e}{f}{\frac{u_m-1}{\delta}}{\frac{\alpha}{\delta}}^{-1}
=
\mtx{e'}{f'}{\frac{u_m-1}{\delta'}}{\frac{\alpha}{\delta'}}
\mtx{\frac{\alpha}{\delta}}{-f}{\frac{1-u_m}{\delta}}{e}=
\mtx{\pm u^{\frac{1}2}}{t}{0}{\pm u^{-\frac{1}{2}}}
\]
holds for some $u\in\Lambda_\lambda$ and $t\in\tk_\lambda$. It follows
that $\delta'\,^2=u\delta^2$, hence $N(\delta^2)=N(\delta'^2)$.

Now 
\[
\sigma_\lambda^{-1}\gamma_0^{-1}\sigma_\kappa\nu
=\mtx{*}{*}{-\frac{(1-u_m)a-\alpha c}{\delta}}{*}\in\sigma_\lambda^{-1}\Gamma\sigma_{\kappa'},
\]
and choosing $\Delta:=\sigma_\lambda^{-1}\Gamma\sigma_{\kappa'}$ in
Lemma $2.9_1$ of \cite{freitag} we get that $N((1-u_m)a-\alpha c)^2\geq N(\delta^2)C_0$
for some positive constant $C_0>0$ (assuming that $(1-u_m)a-\alpha c\neq0$).

Next we handle the cases where one of  the factors $N(c)^2$ and $N((1-u_m)a-\alpha c)^2$ is $0$.
Since 
$\sigma=\sigma_\kappa^{-1}\gamma\sigma_{\kappa}$ for some $\gamma\in\Gamma$, hence
$c=0$ holds if and only if
\[
\infty = \nu\infty=\sigma\sigma_\kappa^{-1}\sigma_{\kappa'}\infty=
\sigma_\kappa^{-1}\gamma\kappa',
\]
that is, $\kappa=\sigma_\kappa\infty=\gamma\kappa'$. But $\Gamma$ permutes the
elements of the class of $\kappa'$, so $c=0$ can hold only if $\kappa=\kappa'$,
and then $\nu=\sigma$, and therefore $\infty=\sigma\infty$ holds. Consequently, the condition $Y_0^{\kappa'}(\sigma_\kappa\sigma^{-1}z)\geq A$
reduces to $Y_0(\sigma^{-1}z)=Y_0(z)\geq A$ in this case.

Assume now that $(1-u_m)a-\alpha c=0$, that is,
$\frac{a}{c}=q$ holds.
This means that 
\[
\sigma_\kappa^{-1}\gamma_0\lambda=q=\nu\infty=\sigma_\kappa^{-1}\gamma\kappa',
\]
hence $\kappa'=\lambda$ must hold.
This means that the exceptional set $\sigma (\sigma_\kappa^{-1} F^*_A)$
can be reduced to
\[
\{
z\in F_{C({\gamma_{m,\alpha}^\kappa})}:\,Y_0(\sigma_\lambda^{-1}\sigma_\kappa\sigma^{-1}z)\geq A 
\}.
\]
The element $\sigma_\lambda^{-1}\sigma_\kappa\sigma^{-1}$ takes $q$ to $\infty$, hence - as we have
 already seen above - it is of the form 
$\mtx{e}{f}{\frac{u_m-1}{\delta}}{\frac{\alpha}{\delta}}$,
where $\delta=e\alpha+f(1-u_m)$,
and then 
\[
Y_0(\sigma_{\lambda}^{-1}\sigma_\kappa\sigma^{-1} z) =
\frac{N(\delta ^2)Y_0(z)}
{\prod_{k=1}^n\abs{(u_m^{(k)}-1)z_k+\alpha_k}^2}=
\frac{N(\delta E_m^{-1})^2Y_0(z)}{\prod_{k=1}^n\abs{z_k-q_k}^2}.
\]

All this shows that the expression in (\ref{hpint}) can be written as
\[
\int_{S_A}k(z,\gamma_{m,\alpha}^\kappa z)u(\sigma_\kappa z)\,d\mu(z),
\]
where
\[
S_A =\left\{ z\in F_{C(\gamma_{m,\alpha}^\kappa)} :\, Y_0(z)\leq A,\,
\frac{N(\delta E_m^{-1})^2Y_0(z)}{\prod_{k=1}^n\abs{z_k-q_k}^2}
\leq A\right\},
\]
at least when $A$ is big enough.
Recall that the centralizer $C(\gamma_{m,\alpha}^\kappa)$ and 
its fundamental domain was described in Section \ref{hyp-par_sec}.
In the following we also use some relating notations defined there.
A direct calculation shows now that the last integral above is
\begin{align*}
\int_{S_A}
\psi
&\left(
\frac{\abs{z_1-q_1 }^2}{(E_m^{(1)})^{-2}y_1^2},\dots,
\frac{\abs{z_n-q_n}^2}{(E_m^{(n)})^{-2}y_n^2}
\right)
u(\sigma_\kappa z)\,d\mu(z)=\\[3mm]
&\qquad\qquad\qquad=\int\limits_{S_A-q}
\psi
\left(
\frac{|E_m^{(1)}z_1|^2}{y_1^2},\dots,
\frac{|E_m^{(n)}z_n|^2}{y_n^2}
\right)
u(\sigma_\kappa(z+q))\,d\mu(z),
\end{align*}
where
\[
S_A-q =\left\{ z\in F_C :\, Y_0(z)\leq A,\,
\frac{N(\delta E_m^{-1})^2Y_0(z)}{\prod_{k=1}^n\abs{z_k}^2}
\leq A\right\}.
\]

The two inequalities above can be written in terms of the polar coordinates as follows:
\[
A_\tht:=\frac{N(\delta E_m^{-1})^2\prod_{k=1}^n\cos\tht_k}{A}\leq\prod_{k=1}^nr_k\leq\frac{A}{\prod_{k=1}^n\cos\tht_k}=:A^\tht.
\]
Since $\displaystyle{\frac{dx\,dy}{y^2}=\frac{dr\,d\tht}{r\cos^2\tht}}$,
the integral above becomes after
the change of variables 
\begin{equation}\label{hyp_par_full_int}
\int\limits_{-\pi/2}^{\pi/2}\dots
\int\limits_{-\pi/2}^{\pi/2}
\psi\left(
    \frac{(E_m^{(1)})^2}{\cos^2\tht_1},\dots,\frac{(E_m^{(n)})^2}{\cos^2\tht_n}
\right)
I_u(A,\tht,m,\alpha)
\prod_{k=1}^n\frac{d\tht_k}{\cos^2\tht_k}
    \end{equation}
where
\[
I_u(A,\tht,m,\alpha)=
\int\limits_{
\begin{array}{c}
\scriptstyle{\log r\in P_{m,\alpha}^\kappa}\\
\scriptstyle{    A_\tht \leq Nr\leq A^\tht}
\end{array}
}
u\left(
    \sigma_\kappa(re^{i(\frac{\pi}{2}+\tht)}+q)
\right)
\prod_{k=1}^n
\frac{dr_k}{r_k}.
\]

We handle the zeroth term and the remaining terms of the
Fourier expansion of the function $u(\sigma_\kappa(z+q))$ 
separately. To this end, for any cusp $\kappa'$ we write 
\[
u(\sigma_{\kappa'} z)=\eta_{\kappa'} y_1^{s_1}\dots y_n^{s_n}
+\phi_{\kappa'} y_1^{1-s_1}\dots y_n^{1-s_n}+R_{\kappa'}(z)=M_{\kappa'}(z)+R_{\kappa'}(z).
\]
Subtracting the contribution of the zeroth term 
$M_\kappa(re^{i(\frac{\pi}{2}+\tht)}+q)$
from the integral
$I_u(A,\tht,m,\alpha)$ one obtains
\[
\int\limits_{
\begin{array}{c}
\scriptstyle{\log r\in P_{m,\alpha}^\kappa}\\
\scriptstyle{    A_\tht \leq Nr\leq A^\tht}
\end{array}
}
R_\kappa\left(
    re^{i(\frac{\pi}{2}+\tht)}+q
\right)
\prod_{k=1}^n
\frac{dr_k}{r_k}.
\]
However, this integral does not converge as $A\to\infty$, 
but Proposition 
\ref{u_bound_prop}
gives that
it does converge if one integrates only over
$\{\log r\in P_{m,\alpha}^\kappa:1\leq Nr\}$. Note that by
the compact support of 
$\psi$
the coordinates of the vector $\cos \tht$
can be assumed to be bounded 
away from zero and hence
$R_\kappa$ 
in the above integral
 can be bounded
uniformly exponentially
 using
Proposition 
\ref{u_bound_prop}.
To ensure convergence on the other half of the set $\{\log r\in P_{m,\alpha}^\kappa\}$ 
we  will subtract the main term
of $u$ at $q$.
By the $\Gamma$-invariance of $u$ we have
\[
u(\sigma_\kappa(z+q))
=u(\sigma_\lambda(\sigma_\lambda^{-1}\gamma_0^{-1}\sigma_\kappa(z+q)))
=u\left(\sigma_\lambda \left(-e(eq+f)-\frac{(eq+f)^2}{z}\right)\right),
\]
hence (after the substitution $r\mapsto \frac{1}{r}$) the integral of 
$u\left(\sigma_\kappa(re^{i(\frac{\pi}{2}+\tht)}+q)\right)$ over the set
$\{\log r\in P_{m,\alpha}^\kappa:A_\tht\leq Nr\leq 1\}$ 
becomes 
\begin{align*}
&\int\limits_{
\begin{array}{c}
\scriptstyle{\log r\in P_{m,\alpha}^\kappa}\\
\scriptstyle{     1\leq Nr\leq A_\tht^{-1}}
\end{array}
}
\!\!\!\!\!\!
u\left(\sigma_\lambda\left(-e(eq+f)+(eq+f)^2re^{i(\frac{\pi}{2}-\tht)}\right)\right)
\prod_{k=1}^n
\frac{dr_k}{r_k}=\\[-2mm]
&\qquad=
\int\limits_{
\begin{array}{c}
\scriptstyle{\log r\in P_{m,\alpha}^\kappa}\\
\scriptstyle{     1\leq Nr\leq A_\tht^{-1}}
\end{array}
}
\!\!\!\!\!\!\!M_\lambda\left((eq+f)^2re^{i(\frac{\pi}{2}-\tht)}\right)
+R_\lambda\left((eq+f)^2re^{i(\frac{\pi}{2}-\tht)}-e(eq+f)\right)
\prod_{k=1}^n
\frac{dr_k}{r_k}.
\end{align*}
Here we also used the translation invariance of $M_\lambda$. It is now clear that
the integral of the second term above converges, i.e.
\[
u\left(\sigma_\kappa(re^{i(\frac{\pi}{2}+\tht)}+q)\right)
-M_\lambda\left(-\frac{(eq+f)^2}{re^{i(\frac{\pi}{2}+\tht)}}\right)
=u\left(\sigma_\kappa(re^{i(\frac{\pi}{2}+\tht)}+q)\right)
-M_\lambda\left(-\frac{\delta^2E_m^{-2}u_m^{-1}}{re^{i(\frac{\pi}{2}+\tht)}}\right)
\]
is integrable over $\{\log r\in P_{m,\alpha}^\kappa:\, Nr\leq 1\}$. 
A straightforward computation (detailed below) shows the 
same for $M_\kappa(re^{i(\frac{\pi}{2}+\tht)})$ and then consequently for the function
\[
\tilde{u}_{m,\alpha}^\kappa
(re^{i(\frac{\pi}{2}+\tht)})=u\left(\sigma_\kappa(re^{i(\frac{\pi}{2}+\tht)}+q)\right)
-M_\kappa\left(re^{i(\frac{\pi}{2}+\tht)}\right)
-M_\lambda\left(-\frac{\delta^2E_m^{-2}u_m^{-1}}{re^{i(\frac{\pi}{2}+\tht)}}\right).
\]
A similar argument gives that $\tilde{u}_{m,\alpha}^\kappa(re^{i(\frac{\pi}{2}+\tht)})$ is integrable over
$\{\log r\in P_{m,\alpha}^\kappa:1\leq Nr\}$ and hence 
over the whole set $\{\log r\in P_{m,\alpha}^\kappa\}$.

Now $I_u(A,\tht,m,\alpha)$ can be written as
\begin{equation}\label{I_sum}
\int\limits_{
\begin{array}{c}
\scriptstyle{\log r\in P_{m,\alpha}^\kappa}\\
\scriptstyle{    A_\tht \leq Nr\leq A^\tht}
\end{array}
}
\!\!\!\!\!\!\!\!\!\!\!\!\!\left[M_\kappa\!\left(re^{i(\frac{\pi}{2}+\tht)}\right)
+M_\lambda\!\left(-\frac{\delta^2E_m^{-2}u_m^{-1}}{re^{i(\frac{\pi}{2}+\tht)}}\right)
\right]
\prod_{k=1}^n
\frac{dr_k}{r_k}
+\!\!
\int\limits_{\log r\in P_{m,\alpha}^\kappa}
\!\!\!\!\!\!\!\!\tilde{u}_{m,\alpha}^\kappa(re^{i(\frac{\pi}{2}+\tht)})
\prod_{k=1}^n
\frac{dr_k}{r_k}+o(A).
\end{equation}
First we turn to the second integral above. 
Observe that the function $U(z):=u(\sigma_\kappa(z+q))$ is invariant under the action
of $\rho_{l_j}$ (defined in (\ref{rho_l_def})):
\begin{align*}
U(\rho_{l_j}z)&
=u(\sigma_{\kappa}(u_{l_j}z+q))
=u(\sigma_{\kappa}(u_{l_j}(z+q)+(1-u_{l_j})q))
\\&=u\left(\sigma_{\kappa}\gamma(l_j)(z+q)\right)=u(\sigma_\kappa(z+q))=U(z),
\end{align*}
because the element $\gamma(l_j)$ is in the centralizer  $C(\gamma_{\alpha,m}^\kappa)\leq 
\sigma_\kappa^{-1}\Gamma\sigma_\kappa$, and $u$ is invariant under the action of $\Gamma$.
The same invariance holds for $M_\kappa(z)$ and 
$M_\lambda\left(-\frac{\delta^2E_m^{-2}u_m^{-1}}{z}\right)$ as well, hence
the invariance of $\tilde{u}_{m,\alpha}^\kappa(z)$ under the action of $\rho_{lj}$  follows.

Let us define the function
\[
F^\kappa_{m,\alpha}(z)=
\int\limits_{\log r\in P_{m,\alpha}^\kappa}\tilde{u}_{m,\alpha}^\kappa(rz)
\prod_{k=1}^n\frac{dr_k}{r_k}.
\]
By the observation of the last paragraph we can see as in the case of mixed elements
that $F^\kappa_{m,\alpha}$ is invariant under coordinate-wise scalar multiplication, i.e.
$F^\kappa_{m,\alpha}(z)=F^\kappa_{m,\alpha}(\tht)$ depends only on 
$\tht$ where $z=re^{i(\frac{\pi}{2}+\tht)}$.

Since the Laplacian $\Delta_k$ is an invariant operator
(i.e. it commutes with the action of $PSL(2,\R)^n$), we get that $U(z)$ is an eigenfunction
of it with the eigenvalue $\lambda_k$. But the same is true $M_\kappa(z)$ and 
$M_\lambda\left(-\frac{\delta^2E_m^{-2}u_m^{-1}}{z}\right)$, and therefore 
$\tilde{u}_{m,\alpha}^\kappa(z)$ and also $F^\kappa_{m,\alpha}(z)$ 
are eigenfunctions
of $\Delta_k$ with the eigenvalue $\lambda_k$. In the same way as in the case of mixed elements
we conclude that the contribution of the second integral of (\ref{I_sum})  in
(\ref{hyp_par_full_int}) is
\begin{align*}
F^\kappa_{m,\alpha}(0,\dots,0)&
\int\limits_{-\frac{\pi}{2}}^{\frac{\pi}{2}}
\dots
\int\limits_{-\frac{\pi}{2}}^{\frac{\pi}{2}}
\psi\left(
    \frac{(E_m^{(1)})^2}{\cos^2\tht_1},\dots,\frac{(E_m^{(n)})^2}{\cos^2\tht_n}
\right)
\left(\prod_{k=1}^n\frac{f_{\lambda_k}(\tht_k)\,d\tht_k}{\cos^2\tht_k}\right),
\end{align*}
where
 $f_{\lambda_k}(\tht)$ is the unique solution of the differential equation
 (\ref{mixed_diff_eq})
with the initial condition $f_{\lambda_k}(0) = 1$ and $f'_{\lambda_k}(0) = 0$ and
\[
F^\kappa_{m,\alpha}(0,\dots,0) = 
\int\limits_{\log r\in P_{m,\alpha}^\kappa}\tilde{u}(r_1i,\dots,r_ni)
\prod_{k=1}^n\frac{dr_k}{r_k}.
\]

Finally we calculate the first integral in (\ref{I_sum}). The map
$\,r\mapsto 
(\log r_1,\dots,\log r_n)$
maps the set we integrate on to 
\[
\mathcal{F}_A:=
\bigcup_{n^{-\frac{1}{2}}\log A_\tht\leq t\leq n^{-\frac{1}{2}}\log A^\tht}\{t \mathbf{1}\}\times \tilde{P}_{m,\alpha}^\kappa,
\]
and the determinant of its Jacobian  is $\prod_{k=1}^n r_k^{-1}$, 
so the  integral  is
\begin{align}\label{hp_main_term_integral}
\int\limits_{\mathcal{F}_A}
M_\kappa\left(e^{x+i(\frac{\pi}{2}+\tht)}\right)
+M_\lambda\left(-\frac{\delta^2E_m^{-2}u_m^{-1}}{e^{x+i(\frac{\pi}{2}+\tht)}}\right)\,\prod_{k=1}^ndx_k.
\end{align}
The first term of the zeroth Fourier coefficient $M_\kappa$ gives the following contribution to the integral above: 
\begin{align}\label{frst_term_int}
\eta_\kappa\prod_{k=1}^n\left(\cos\tht_k\right)^{s_k}
\int\limits_{\mathcal{F}_A}
\exp\left(\sum_{k=1}^ns_kx_k\right)\,\prod_{k=1}^n\,dx_k.
\end{align}
Let $\mathcal{L}$ be the linear map that maps the standard basis of $\R^n$ to
$\mathbf{1},v_1^\kappa,\dots,v_{n-1}^\kappa$ (see Section \ref{hyp-par_sec} for the
definitions) so that 
$\mathcal{L}([n^{-\frac12}\log A_\tht;n^{-\frac12}\log A^\tht]\times[-\frac12;\frac12]^{n-1})=\mathcal{F}_A$. 
The matrix of this map w.r.t. the standard basis is
$[\mathcal{L}]=\mathcal{E}_\kappa\mtx{n^{-\frac12}}{}{}{L_{m,\alpha}^\kappa}$, where $\mathcal{E}_\kappa$ is defined in (\ref{eps_s_eqs}) and $L_{m,\alpha}^\kappa$ 
is defined in (\ref{L_mtx_def}).
Hence, after a change of variables,
the integral in (\ref{frst_term_int}) becomes
\begin{align}\label{int_calculation}
\abs{\det[\mathcal{L}]}&
\int\limits_{n^{-\frac12}\log A_\tht}^{n^{-\frac12}\log A^\tht}e^{n^{-\frac12}t_0\sum_{k=1}^ns_k}\,dt_0
\prod_{q=1}^{n-1}\int\limits_{-1/2}^{1/2}
\exp\left(t_q\sum_{j=1}^{n-1}l_q^{(j)}\sum_{k=1}^n s_k\log\eps_j^\kappa\,^{(k)}
\right)
\,dt_q=\nonumber\\[3mm]
&=\abs{\det[\mathcal{L}]}
\int\limits_{n^{-\frac12}\log A_\tht}^{n^{-\frac12}\log A^\tht}e^{t_0n^{\frac12}s}\,dt_0
\prod_{q=1}^{n-1}\int\limits_{-1/2}^{1/2}
\exp\left(2\pi it_q\sum_{j=1}^{n-1}l_q^{(j)}m_{u,\kappa}^{(j)}
\right)
\,dt_q
\end{align}
by (\ref{eps_s_eqs}). This expression is zero unless 
$(L_{m,\alpha}^\kappa)^T m_{u,\kappa}=0$ holds.
Since 
\[
0\neq \det[\mathcal{L}]=n^{-\frac12}\det \mathcal{E}_\kappa\det L_{m,\alpha}^\kappa,
\]
i.e.  $\det L_{m,\alpha}^\kappa\neq 0$,
this can hold only if $m_{u,\kappa}=0$. In the latter case (\ref{frst_term_int}) becomes
\[
\frac{\eta_\kappa\abs{\det \mathcal{E}_\kappa}\abs{\det L_{m,\alpha}^\kappa}}{ns}
\left[
A^{s}- 
\frac{N(\delta E_m^{-1})^{2s}\prod_{k=1}^n(\cos\tht_k)^{2s}}{A^{s}}
\right]=
\frac{\eta_\kappa\abs{\det \mathcal{E}_\kappa}\abs{\det L_{m,\alpha}^\kappa}A^s}{ns}+o(A).
\]
Here we used that since $\psi$ has a compact support, the values $\cos\tht_k$ are 
bounded away from zero by a constant depending on $\psi$ and $\Gamma$.
The same argument gives that the contribution of the second term of the zeroth Fourier coefficient
$M_\kappa$ in $I_u(A,\tht,m,\alpha)$ is
\begin{equation}\label{hp_2nd_main_term}
\frac{\phi_\kappa\abs{\det \mathcal{E}_\kappa}\abs{\det L_{m,\alpha}^\kappa}A^{1-s}}{n(1-s)}+o(A).
\end{equation}

Now we turn to the integral of the second term in (\ref{hp_main_term_integral}).
The first term of this Fourier coefficient contributes
\begin{align*}
\eta_\lambda\prod_{k=1}^n
\left((\delta^{(k)})^2(E_m^{(k)})^{-2}(u_m^{(k)})^{-1}\cos\tht_k\right)^{s_k}
\int\limits_{\mathcal{F}_A}
\exp\left(-\sum_{k=1}^ns_kx_k\right)\,\prod_{k=1}^n\,dx_k.
\end{align*}
The same computation as above shows that this is zero unless $m_{u,\kappa}=0$, in which
case one gets
\begin{align*}
\eta_\lambda \abs{\det[\mathcal{L}]}N(\delta^2 E_m^{-2})^{s}&\prod_{k=1}^n(\cos\tht_k)^{s}
\int\limits_{n^{-\frac12}\log (A^\tht)^{-1}}^{n^{-\frac12}\log A_\tht^{-1}}
\!\!\!\!\!\!\!e^{t_0n^{\frac12}s}\,dt_0
=\frac{\eta_\lambda\abs{\det\mathcal{E}_\kappa}\abs{\det L_{m,\alpha}}}{ns}
A^s+o(A).
\end{align*}
Finally, we get one more term which is of the
form as the one in (\ref{hp_2nd_main_term}), replacing $\phi_\kappa$ by $\phi_\lambda$.

We conclude that (\ref{hp_main_term_integral}) contributes in (\ref{hyp_par_full_int})
(aside from an $o(A)$ term) 
 as
\[
\left(
\frac{(\eta_\kappa+\eta_\lambda) A^s}{s}+\frac{(\phi_\kappa+\phi_\lambda) A^{1-s}}{1-s}
\right)
\frac{
\abs{\det  L_{m,\alpha}^\kappa}
\abs{\det \mathcal{E}_\kappa}}{n}
\!\!\!\int\limits_{-\pi/2}^{\pi/2}\!\!\!\!\dots\!\!\!\!
\int\limits_{-\pi/2}^{\pi/2}
\!\!\psi\left(
    \frac{E_m^2}{\cos^2\tht}
\right)
\prod_{k=1}^n\frac{d\tht_k}{\cos^2\tht_k},
\]
where the last integral is
\begin{align*}
&2^n
\int\limits_0^{\frac{\pi}{2}}
\dots
\int\limits_0^{\frac{\pi}{2}}
\psi\left(
    \frac{(E_m^{(1)})^2}{\cos^2\tht_1},\dots,\frac{(E_m^{(n)})^2}{\cos^2\tht_n}
\right)
\prod_{k=1}^n\frac{d\tht_k}{\cos^2\tht_k}=\\[6mm]
&=
\frac{1}{N(\abs{E_m})}
\int\limits_{(E_m^{(1)})^2}^{\infty}
\dots
\int\limits_{(E_m^{(n)})^2}^{\infty}
\frac{\psi(t_1,\dots,t_n)}{\prod\limits_{k=1}^n\sqrt{t_k-(E_m^{(k)})^2}}
\,\prod_{k=1}^n\,dt_k
=
\frac{1}{N(\abs{E_m})}
g(\log u_m^{(1)},\dots,\log u_m^{(n)}).
\end{align*}

We summarize the results obtained so far.
 The contribution
of the hyperbolic-parabolic
classes in the truncated trace is
\begin{equation}\label{hyp_par_partial_result}
\Sigma_{\mathrm{hyp-par}}
=\frac{1}{2}\sum_{\kappa\in\mathcal{S}}\,\,\sum_{m\in\Z^{n-1}\setminus \{0\}}
\,\,\sum_{\alpha\in \tk_\kappa^m/\Lambda_\kappa}\,
\left[\delta_{m_{u,\kappa}}M_\kappa(m,\alpha,A)+C_\kappa(m,\alpha)\right]+o(A),
\end{equation}
where 
the main term $M_\kappa(m,\alpha,A)$ is given by
\[
\left(
\frac{(\eta_\kappa+\eta_{\tilde{\kappa}_{m,\alpha}}) A^s}{s}
+\frac{(\phi_\kappa+\phi_{\tilde{\kappa}_{m,\alpha}}) A^{1-s}}{1-s}
\right)
\frac{
\abs{\det L_{m,\alpha}^\kappa}
\abs{\det \mathcal{E}_\kappa}}{n N(\abs{E_m^\kappa})}
g(\log u_m^\kappa),
\]
where $\tilde{\kappa}_{m,\alpha}\in\mathcal{S}$ is the cusp for $\Gamma$ that can be
taken
 (by an element of $\Gamma$) to $\sigma_\kappa \frac{\alpha}{1-u_m^\kappa}$,
and the term $C_\kappa(m,\alpha)$ is
\begin{align}\label{hyp_par_const_terms}
\int\limits_{\log r\in P_{m,\alpha}^\kappa}\tilde{u}_{m,\alpha}^\kappa(ri)
\prod_{k=1}^n\frac{dr_k}{r_k}&
\int\limits_{-\frac{\pi}{2}}^{\frac{\pi}{2}}
\dots
\int\limits_{-\frac{\pi}{2}}^{\frac{\pi}{2}}
\psi\left(
    \frac{(E_m^{\kappa})^2}{\cos^2\tht}
\right)
\left(\prod_{k=1}^n\frac{f_{\lambda_k}(\tht_k)\,d\tht_k}{\cos^2\tht_k}\right),
\end{align}
Moreover, the terms above are zero for any cusp $\kappa$  for all but finitely many $m$.

Let us fix a $\kappa\in\mathcal{S}$, an $m\in\Z^{n-1}\setminus \{0\}$ and an 
$\alpha\in\tk_\kappa^m/\Lambda_\kappa$  in the sum (\ref{hyp_par_partial_result}) above. The corresponding term
is counted twice, it occurs also in the case of the cusp 
$\kappa'=\tilde{\kappa}_{m,\alpha}$ 
for an appropriate $m'\in\Z^{n-1}\setminus\{0\}$ and a class $\beta$. 
It follows that the main terms $\delta_{m_{u,\kappa}}M_\kappa(m,\alpha,A)$ and 
$\delta_{m_{u,\kappa'}}M_{\kappa'}(m',\beta,A)$ are equal.

At this point we specify the function $u(z)$, namely, we work with
the Eisenstein series $E_\kappa(z,s,0)$ for some fixed $\frac{1}{2}<s<1$ 
(defined in (\ref{Eisenstein_def})).
It is not a cusp form and $s_1=\ldots=s_n$ hold for its eigenvalues, and therefore
$m_{u,\kappa}=m_{u,\kappa'}=0$ must hold by (\ref{eps_s_eqs}) which implies 
$M_\kappa(m,\alpha,A)=M_{\kappa'}(m',\beta,A)$. 
Also, since $\eta_\kappa=1$ and $\eta_{\kappa'}=0$, the first factor
of each of these main terms is non-zero.

Assume that $u_m^\kappa\neq u_{m'}^{\tilde{\kappa}_{m,\alpha}}$ holds,
then we can choose the function $g$ so that exactly one of 
$g(\log u_m^\kappa)$ and $g(\log u_m^{\tilde{\kappa}_{m,\alpha}})$ is zero.
This  yields that exactly one of  $M_\kappa(m,\alpha,A)$ and 
$M_{\tilde{\kappa}_{m,\alpha}}(m',\beta,A)$ is zero, which is impossible,
and hence $u_m^\kappa=u_{m'}^{\tilde{\kappa}_{m,\alpha}}$ must hold.

Let us choose the vector $m=\mathbf{e}_j=(0,\dots,0,1,0,\dots,0)^T$ whose the $j$th
coordinate
is $1$ and the others are zero ($1\leq j\leq n-1$).
Then $u_{\mathbf{e}_j}^\kappa=\eps_j^\kappa$, and it follows that
$\eps_j^\kappa\in\Lambda_{\tilde{\kappa}_{m,\alpha}}$, i.e.
$\Lambda_\kappa\subset \Lambda_{\tilde{\kappa}_{m,\alpha}}$.
Changing the role of $\kappa$ and $\tilde{\kappa}_{m,\alpha}$ in the previous argument
and specifying $m'$ instead of $m$ we infer that 
$\Lambda_{\tilde{\kappa}_{m,\alpha}}\subset \Lambda_\kappa$ hence
these groups are identical. Hence to conclude the proof of Proposition \ref{multi_group_prop}., that is, to show that the multiplier group 
$\Lambda_\kappa$ is independent of $\kappa$  it is enough to prove
the following:
\begin{lemma}
    For any two different cusps $\kappa$, $\kappa'$ there is
 a hyperbolic-parabolic  element in $\Gamma$ with fixed points $\kappa$ and $\kappa'$.
\end{lemma}
\begin{proof}
    In the first step we show the analogous statement for the Hilbert modular group 
    $\Gamma_K$
    and any two different cusps $\kappa$ and $\kappa'$ for $\Gamma_K$.
These cusps can and will be represented by an element of the field $K$ and the corresponding
vector is obtained via the different embeddings of $K$ into $\R$.
     It is well-known that
the number of the equivalence classes of cusps for $\Gamma_K$ is the class number 
$h=h(K)$ of $K$ (see Proposition 20 on page 188 in \cite{siegel}).
These classes are represented by a fixed set of integer ideals
$\mathfrak{a}_1,\dots,\mathfrak{a}_h\subset \OO_K$ such that the corresponding
cusps are written in the form $\lambda_j=\rho_j/\sigma_j$ where $\rho_j,\sigma_j\in\OO_K$
and $\mathfrak{a}_j=(\rho_j, \sigma_j)$.

Assume first that $\kappa=\lambda_j$ for some $1\leq j\leq h$.
 Let us fix the elements $\eta_j,\xi_j\in\mathfrak{a}_j^{-1}$ such that
$\rho_j\eta_j-\xi_j\sigma_j=1$ holds, then the matrix
\begin{align*}
A_j=\mtx{\rho_j}{\xi_j}{\sigma_j}{\eta_j}\in SL(2,K)
\end{align*}
takes $\infty$ to $\lambda_j$, hence $\infty$ is a cusp of
$A_j^{-1}\Gamma_K A_j$.
 The stabilizer of $\infty$ in this group consists of
 elements of the form
 \[
 \mtx{u}{\zeta u^{-1}}{0}{u^{-1}}
 \]
 with $u\in\OO_K^\times$ and $\zeta\in \id{a}_j^{-2}$ (see \cite{siegel}). 
 For a $\kappa'\in (K\cup\{\infty\})\setminus \{\kappa\}$, let
 $c=A_j^{-1}\kappa'$ be a cusp of $A_j^{-1}\Gamma_K A_j$ different from
 $\infty$. We show that
 the latter matrix above can be chosen so that its other fixed cusp is $c$.
For this, it is enough to choose the unit $u$ so that $c(1-u^2)\in\mathfrak{a}_j^{-2}$
holds. But this can be reached since
for an arbitrary integral ideal $\mathfrak{a}$ one can choose $u$ so that
$\mathfrak{a}\mid (1-u^2)\Longleftrightarrow 1-u^2\in\mathfrak{a}$ holds.

It follows that there is a hyperbolic-parabolic element in $\Gamma$ with
fixed points $\lambda_j$ and $\kappa'$. But any cusp $\kappa$ can be written
as $\gamma_\kappa^{-1}\lambda_j$ for some $j$ and $\gamma_\kappa^{-1}\in\Gamma$,
so if $\gamma\in\Gamma$ is a hyperbolic-parabolic element that fixes $\lambda_j$ and
$\gamma_\kappa \kappa'$, then $\gamma_\kappa^{-1}\gamma\gamma_\kappa$ fixes $\kappa$
and $\kappa'$.
Since the cusps 
are the same for any finite index subgroup of $\Gamma_K$ (though their equivalence classes are not), the claim of the lemma follows now from Theorem \ref{1_3_commensurable_thm}.
\end{proof}
From now on, we drop the index in the notation of the multiplier group and
assume that the generating set $\eps_1^{\kappa},\dots,\eps_{n-1}^{\kappa}$
is the same for any $\kappa\in\mathcal{S}$. 
Hence the matrices $\mathcal{E}_\kappa$ are identical
for any $\kappa$, so we omit the indices here as well.
Note that the integer vectors $m_{u,\kappa}$ are also defined in terms of 
$\mathcal{E}$ and therefore their common value will be denoted by $m_u$.

Returning to the main terms $M_\kappa(m,\alpha,A)$ and
$M_{\tilde{\kappa}_{m,\alpha}}(m',\beta,A)$ in our argument,
from  $u_m^\kappa= u_{m'}^{\tilde{\kappa}_{m,\alpha}}$ we infer that
$m=m'$ and we simply write $u_m$ and $E_m$ in the following.
Finally, the equality of the main terms implies that
$|\det L_{m,\alpha}^\kappa|=
|\det L_{m,\beta}^{\tilde{\kappa}_{m,\alpha}}|$. 

Now we return to a general form $u$, and we only assume that it is
not a cusp form and $m_u=0$ holds (otherwise
there are no main terms in (\ref{hyp_par_partial_result})).
We split each main term $M_\kappa(m,\alpha,A)$ into two parts:
\begin{align*}
&\left(
\frac{\eta_\kappa A^s}{s}
+\frac{\phi_\kappa A^{1-s}}{1-s}
\right)
\frac{
\abs{\det L_{m,\alpha}^\kappa}
\abs{\det \mathcal{E}}}{n N(\abs{E_m})}
g(\log u_m)+
\\[3mm]
&\qquad\qquad+
\left(
\frac{\eta_{\tilde{\kappa}_{m,\alpha}} A^s}{s}
+\frac{\phi_{\tilde{\kappa}_{m,\alpha}} A^{1-s}}{1-s}
\right)
\frac{
\abs{\det L_{m,\alpha}^\kappa}
\abs{\det \mathcal{E}}}{n N(\abs{E_m})}
g(\log u_m).
\end{align*}
This main term has the (equal) pair $M_{\tilde{\kappa}_{m,\alpha}}(m,\beta,A)$,
and if $\lambda = \tilde{\kappa}_{m,\alpha}$ then clearly 
 $\tilde{\lambda}_{m,\beta}=\kappa$
 and
 the pair of $M_{\lambda}(m,\beta,A)$
 is $M_\kappa(m,\alpha,A)$. That is, this pairing
 gives a bijection for every $m$ on the set of pairs
 $(\kappa,\alpha)$, where $\kappa\in\mathcal{S}$ and 
 $\alpha\in\tk_\kappa^m/\Lambda$. Moreover,
 the term $M_{\tilde{\kappa}_{m,\alpha}}(m,\beta,A)$
 has the split form
 \begin{align*}
&\left(
\frac{\eta_{\tilde{\kappa}_{m,\alpha}} A^s}{s}
+\frac{\phi_{\tilde{\kappa}_{m,\alpha}} A^{1-s}}{1-s}
\right)
\frac{
|\det L_{m,\beta}^{\tilde{\kappa}_{m,\alpha}}|
\abs{\det \mathcal{E}}}{n N(\abs{E_m})}
g(\log u_m)+
\\[3mm]
&\qquad\qquad+
\left(
\frac{\eta_{\kappa} A^s}{s}
+\frac{\phi_{\kappa} A^{1-s}}{1-s}
\right)
\frac{
|\det L_{m,\beta}^{\tilde{\kappa}_{m,\alpha}}|
\abs{\det \mathcal{E}}}{n N(\abs{E_m})}
g(\log u_m).
\end{align*}
By the last remark of the previous paragraph we
have that the first  term of $M_\kappa(m,\alpha,A)$
is the second term of $M_{\tilde{\kappa}_{m,\alpha}}(m,\beta,A)$
and vice versa. 
These simple observations imply immediately that if
we sum the main terms in (\ref{hyp_par_partial_result}) obtaining
\begin{align*}
\frac{\delta_{m_{u}}\abs{\det \mathcal{E}}}{2n}
&\left[
\sum_{\kappa\in\mathcal{S}}
\left(
\frac{\eta_\kappa A^s}{s}
+\frac{\phi_\kappa A^{1-s}}{1-s}
\right)
\sum_{m\in\Z^{n-1}\setminus \{0\}}
\frac{g(\log u_m)}{N(\abs{E_m})}
\sum_{\alpha\in \tk_\kappa^m/\Lambda}\,\,
\abs{\det L_{m,\alpha}^\kappa}
\right.\\[3mm]
&\quad+\left.
\sum_{\kappa\in\mathcal{S}}\,\,\sum_{m\in\Z^{n-1}\setminus \{0\}}
\frac{g(\log u_m)}{N(\abs{E_m})}
\sum_{\alpha\in \tk_\kappa^m/\Lambda}
\left(
\frac{\eta_{\tilde{\kappa}_{m,\alpha}} A^s}{s}
+\frac{\phi_{\tilde{\kappa}_{m,\alpha}} A^{1-s}}{1-s}
\right)
\abs{\det L_{m,\alpha}^\kappa}
\right],
\end{align*}
then the two triple sums above are equal, hence this expression  simply becomes
\begin{equation}\label{pre_hyp-par_main_term}
\frac{\delta_{m_{u}}\abs{\det \mathcal{E}}}{n}
\left[
\sum_{\kappa\in\mathcal{S}}
\left(
\frac{\eta_\kappa A^s}{s}
+\frac{\phi_\kappa A^{1-s}}{1-s}
\right)
\sum_{m\in\Z^{n-1}\setminus \{0\}}
\frac{g(\log u_m)}{N(\abs{E_m})}
\sum_{\alpha\in \tk_\kappa^m/\Lambda}\,\,
\abs{\det L_{m,\alpha}^\kappa}
\right].
\end{equation}

Next we give group theoretic interpretations of the 
quantities $N(|E_m|)$ and $|\det L_{m,\alpha}^\kappa|$.
The sublattice $(u_m-1)\tk_\kappa$ of $\tk_\kappa$ is 
obtained by coordinate-wise multiplication, i.e. via multiplication by a diagonal matrix 
with entries $(u_m-1)^{(k)}$ in its diagonal.
It is well-known that the index of this sublattice in $\tk_\kappa$, i.e.
the order of the factor group $\tk_\kappa^m=\tk_\kappa/(u_m-1)\tk_{\kappa}$
is the absolute value of the 
determinant of this matrix, that is simply $|N(u_m-1)|=N(|E_m|)$.

Now we consider the $\Lambda$-equivalent elements of $\tk_\kappa^m$.
Assume that for an $\alpha\in\tk_\kappa^m$
we have $u_l\alpha=\alpha$ (in $\tk_\kappa^m$) for some $l\in\Z^{n-1}$.
This means exactly that $\frac{u_{l}-1}{u_m-1}\alpha\in\tk_\kappa$.
A simple computation shows that in this case the element
\[
\gamma(l,m,\alpha):=\mtx{u_l^{1/2}}{\frac{u_l-1}{u_m-1}\alpha u_l^{-1/2}}{0}{u_l^{-1/2}}\in 
C(\gamma^\kappa_{m,\alpha}).
\]
Every element of the centralizer $C(\gamma^\kappa_{m,\alpha})$
has this form by Proposition \ref{hyp_par_centralizer}, and 
hence $u_l\alpha=\alpha$ is equivalent to 
$\gamma(l,m,\alpha)\in C(\gamma^\kappa_{m,\alpha})$. Again, by Proposition \ref{hyp_par_centralizer} this holds if and only if
$\log u_l$ is in the lattice spanned by $v_1^\kappa,\dots,v_{n-1}^\kappa$
in the  subspace $V=\{a\in\R^n:a_1+\dots+a_n=0\}$, where
$v_j^\kappa=l_j^{(1)}\log \eps_1+\dots+l_j^{(n-1)}\log\eps_{n-1}$ and
the integer vectors $l_j\in\Z^{n-1}$ are defined in Proposition
\ref{hyp_par_centralizer}.
Hence, for a fixed $\alpha$, the number of inequivalent  points $u_l\alpha\in\tk_\kappa^m$ 
is exactly the index of the before-mentioned sublattice in the
lattice generated by $\log\eps_1,\dots,\log\eps_{n-1}$ in $V$, and this is
exactly $|\det L_{m,\alpha}^\kappa|$. It follows that
\[
\sum_{\alpha\in\tk_\kappa^m/\Lambda}|\det L_{m,\alpha}^\kappa|=|\tk_\kappa^m|=N(|E_m|),
\]
and hence (\ref{pre_hyp-par_main_term}) becomes
\[
\frac{\delta_{m_{u}}\abs{\det \mathcal{E}}}{n}
\sum_{\kappa\in\mathcal{S}}
\left(
\frac{\eta_\kappa A^s}{s}
+\frac{\phi_\kappa A^{1-s}}{1-s}
\right)
\sum_{m\in\Z^{n-1}\setminus \{0\}}
g(\log u_m)
\]
and this (together with (\ref{hyp_par_partial_result}) and (\ref{hyp_par_const_terms})) completes the proof of Theorem \ref{hyp-par_thm}.

\subsection{Extension of \texorpdfstring{$\zeta$}{zeta}-functions corresponding to lattices}

In this section we prove Lemma \ref{zeta_lemma}.
We use the notations of Section \ref{zeta_section} and note that
the following argument is a standard one in analytic number theory 
(we basically copy the proof of Theorem 1.7.2 in \cite{bump1998automorphic})
and hence some details
will be omitted.

It is easy to see that the sum in (\ref{zeta_def}) converges absolutely and locally
uniformly for $\ree s>1$, and this latter condition 
will be assumed in the first part of the proof.
We write the terms of the sum in (\ref{zeta_def})
as follows:
let $0\neq l\in L$, $l=(l^{(1)},\dots,l^{(n)})$, then
\[
\int_0^\infty e^{-\pi x(l^{(k)})^2}x^{\frac{s_k}{2}}\,\frac{dx}{x}=
\frac{1}{\pi^{\frac{s_k}{2}}(l^{(k)})^{s_k}}\int_0^\infty e^{-u}u^{\frac{s_k}{2}}\,\frac{du}{u}
=\frac{\Gamma\left(\frac{s_k}{2}\right)}{\pi^{\frac{s_k}{2}}\abs{l^{(k)}}^{s_k}}
\]
and hence
\[
\int_{(\R^+)^n}
e^{-\pi\tr(\mathbf{x} l^2)}\prod_{k=1}^n x_k^{\frac{s_k}{2}}\,\frac{dx_k}{x_k}
=\pi^{-\frac{ns}{2}}\frac{\lambda_{M,-m}(\abs{l})}{\abs{N l}^s}
\prod_{k=1}^n\Gamma\left(\frac{s_k}{2}\right),
\]
where $\mathbf{x}=(x_1,\dots,x_n)^T$, $\mathbf{x} l^2$ is the coordinate-wise product of $\mathbf{x}$ and $l^2$,
and the trace  $\tr(\cdot)$ of a vector is  the sum of its coordinates. Then
\begin{align*}
\pi^{-\frac{ns}{2}}\left(\prod_{k=1}^n\Gamma\left(\frac{s_k}{2}\right)\right)Z_{L,M}(s,m)&=
\sum_{0\neq l\in L/M}
\int_{(\R^+)^n}e^{-\pi\tr(\mathbf{x}l^2)}\prod_{k=1}^n x_k^{\frac{s_k}{2}}\,\frac{dx_k}{x_k}\\[3mm]
&=\sum_{\eps\in M^2}\int_{(\R^+)^n/M^2}
\left(\sum_{0\neq l \in L/M}e^{-\pi\tr(\eps \mathbf{x}l^2)}\right)
\prod_{k=1}^n (\eps^{(k)}x_k)^{\frac{s_k}{2}}\,\frac{dx_k}{x_k}\\[3mm]
&=\int_{(\R^+)^n/M^2}\left(\sum_{0\neq l\in L}e^{-\pi\tr(\mathbf{x}l^2)}\right)
\prod_{k=1}^n x_k^{\frac{s_k}{2}}\,\frac{dx_k}{x_k}.
\end{align*}
Let us define the following theta function for the lattice $L$:
\[
\Theta_L(\mathbf{x}):=\sum_{l\in L} e^{-\pi \tr(\mathbf{x}l^2)}.
\]
Using this notation we have
\begin{align}\label{XI_integral}
\Xi_{L,M}(s,m):=
\pi^{-\frac{ns}{2}}\left(\prod_{k=1}^n\Gamma\left(\frac{s_k}{2}\right)\right)Z_{L,M}(s,m)
=\int_{(\R^+)^n/M^2}(\Theta_L(\mathbf{x})-1)
\prod_{k=1}^n x_k^{\frac{s_k}{2}}\,\frac{dx_k}{x_k}.
\end{align}

For a fixed $\mathbf{x}\in(\R^+)^n$ we set $f_\mathbf{x}(\mathbf{y})=e^{-\pi\tr(\mathbf{x}\mathbf{y}^2)}$,
its Fourier transform is
\begin{align*}
\hat{f}_\mathbf{x}(\xi)&=\int_{\R^n}f_\mathbf{x}(\mathbf{y})
e^{-2\pi i\left<\mathbf{y},\mathbf{\xi}\right>}\,\mathrm{d}\mathbf{y}=
\prod_{k=1}^n\frac{1}{\sqrt{x_k}} e^{-\pi \xi_k^2/x_k}.
\end{align*}
By the Poisson summation formula we have
\begin{align}\label{tht_eq}
\Theta_L(\mathbf{x})&=\sum_{l \in L}f_\mathbf{x}(l)=
\frac{1}{\mathrm{vol}(\R^n/L)}\sum_{\beta\in L^*}\hat{f_\mathbf{x}}(\beta)
=\frac{1}{\mathrm{vol}(\R^n/L) (N\mathbf{x})^{1/2}}\Theta_{L^*}(1/\mathbf{x}).
\end{align}
Now we split the integral in (\ref{XI_integral}) into two parts depending on $N\mathbf{x}$.
If $N\mathbf{x}<1$, we use (\ref{tht_eq}) and then substitute $1/\mathbf{x}$
to obtain that $\Xi_{L,M}(s,m)$ is
\begin{align*}
&  
\int\limits_{\begin{array}{c}    
    \scriptstyle{(\R^+)^n/M^2}\\
    \scriptstyle{N\mathbf{x}>1}
\end{array}}
(\Theta_L(\mathbf{x})-1)\prod_{k=1}^n x_k^{\frac{s_k}{2}}\,\frac{dx_k}{x_k}
+\dfrac{1}{\mathrm{vol}(\R^n/L) }
\int\limits_{\begin{array}{c} 
    \scriptstyle{(\R^+)^n/M^2}\\
    \scriptstyle{N\mathbf{x}>1}
\end{array}}
(\Theta_{L^*}(\mathbf{x})-1)\prod_{k=1}^n x_k^{\frac{1-s_k}{2}}\,\frac{dx_k}{x_k}\\
&\qquad-
\int\limits_{\begin{array}{c} 
    \scriptstyle{(\R^+)^n/M^2}\\
    \scriptstyle{N\mathbf{x}<1}
\end{array}}
\prod_{k=1}^n x_k^{\frac{s_k}{2}}\,\frac{dx_k}{x_k}
+\dfrac{1}{\mathrm{vol}(\R^n/L)}
\int\limits_{\begin{array}{c} 
    \scriptstyle{(\R^+)^n/M^2}\\
    \scriptstyle{N\mathbf{x}>1}
\end{array}}
\prod_{k=1}^n x_k^{\frac{1-s_k}{2}}\,\frac{dx_k}{x_k}.
\end{align*}
A straightforward computation (similar to the one that led to (\ref{int_calculation})
in the previous proof) shows that
\begin{align*}
\int\limits_{\begin{array}{c} 
    \scriptstyle{(\R^+)^n/M^2}\\
    \scriptstyle{N\mathbf{x}<1}
\end{array}}
\prod_{k=1}^n x_k^{\frac{s_k}{2}}&\,\frac{dx_k}{x_k}
=2^{n-1}\abs{\det \mathcal{E}_M}
\int_{-\infty}^0e^{nsy_0/2}\,dy_0\prod_{j=1}^{n-1}\int_0^1e^{2\pi m_j iy_j}\,dy_j.
\end{align*}
This expression is $0$ unless all coordinates of $m$ are zero, in which case
it is $2^n\abs{\det \mathcal{E}_M}/(ns)$.
Similarly, 
\begin{align*}
\int\limits_{\begin{array}{c} 
    \scriptstyle{(\R^+)^n/M^2}\\
    \scriptstyle{N\mathbf{x}>1}
\end{array}}
\prod_{k=1}^n &x_k^{\frac{1-s_k}{2}}\,\frac{dx_k}{x_k} 
=2^{n-1}\abs{\det \mathcal{E}_M}
\int_{0}^\infty e^{n(1-s)y_0/2}\,dy_0\prod_{j=1}^{n-1}\int_0^1e^{-2\pi m_j iy_j}\,dy_j,
\end{align*}
and as above, this is $0$ if $m\neq0$ and otherwise we get
$\frac{2^{n}\abs{\det \mathcal{E}_M}}{n(s-1)}$ (note that $\ree s>1$ is still assumed here).

Therefore if $m\neq 0$, then $\Xi_{L,M}(s,m)$ is entire and
\begin{equation}\label{func_eq}
\mathrm{vol}(\R^n/L)^{1/2}\Xi_{L,M}(s,m)=\mathrm{vol}(\R^n/L^*)^{1/2}\Xi_{L^*,M}(1-s,-m)
\end{equation}
holds. If $m=0$, then $\Xi_{L,M}(s,m)$ is holomorphic except for $s=1$ and $s=0$, where
it has simple poles with residues 
$\frac{2^n\abs{\det \mathcal{E}_M}}{n\cdot \mathrm{vol}(\R^n/L)}$ and 
$-\frac{2^n\abs{\det \mathcal{E}_M}}{n}$, respectively. The functional equation (\ref{func_eq})  holds also
in this case for any $s\neq0,1$.

One can reorder the equation (\ref{func_eq}) asymmetrically:
\[
Z_{L,M}(s,m)
=
\mathrm{vol}(\R^n/L^*)
\pi^{ns-\frac{n}{2}}\left(\prod_{k=1}^n
\frac{\Gamma\left(\frac{1-s_k}{2}\right)}{\Gamma\left(\frac{s_k}{2}\right)}\right)
Z_{L^*,M}(1-s,-m).
\]
Here 
\[
\frac{\Gamma\left(\frac{1-s_k}{2}\right)}{\Gamma\left(\frac{s_k}{2}\right)}=
\Gamma\left(\frac{1-s_k}{2}\right)\Gamma\left(1-\frac{s_k}{2}\right)
\frac{\sin\frac{\pi s_k}{2}}{\pi}=
2^{s_k}\Gamma\left(1-s_k\right)
\frac{\sin\frac{\pi s_k}{2}}{\pi^{\frac{1}{2}}},
\]
i.e.
\[
Z_{L,M}(s,m)
=
\mathrm{vol}(\R^n/L^*)2^{ns}
\pi^{n(s-1)}\left(\prod_{k=1}^n
\Gamma\left(1-s_k\right)\sin\frac{\pi s_k}{2}
\right)
Z_{L^*,M}(1-s,-m).
\]
If $\ree s$ is bounded and $\abs{\im s}\geq t_0$ for some big enough $t_0>0$ 
then by Stirling's formula
\[
\abs{\Gamma(1-s_k)\sin\frac{\pi s_k}{2}}\asymp\abs{\im s}^{1/2-\ree s}
\]
and hence
$\abs{Z_{L,M}(s,m)}\asymp(\im s)^{n\left(\frac{1}{2}-\ree s\right)}
\abs{Z_{L^*,M}(1-s,-m)}$.
By the Phragmén-Lindelöf principle, it follows  from this and from
the trivial bound $\abs{Z_{L,M}(s,m)}\leq \abs{Z_{L,M}(\ree s,0)}$ for
$\ree s>1$ that
\[
Z_{L,M}(s,m)\ll_{\eps,m}\abs{\im s}^{n(1-\ree s)/2+\eps}
\]
holds for $0\leq \ree s\leq 1$ and $\abs{\im s}\geq t_0>0$. Similarly,
one can bound $Z_{L,M}(s,m)$ by $\abs{\im s}^\eps$ once $\ree s>1$ and $\abs{\im s}\geq t_0$
(or by a constant
if $\ree s>1+\delta$ for some $\delta>0$). This completes the proof of Lemma \ref{zeta_lemma}.

\subsection{Proof in the totally parabolic case}

We proceed by calculating the part of the trace where we sum over parabolic
classes. 
Every such class is represented by an element that fixes a cusp $\kappa\in\mathcal{S}$.
An  element of this type is conjugated by $\sigma_\kappa\in \psl$ to an element of the form
\[
\gamma_\alpha^\kappa :=\mtx{1}{\alpha}{0}{1},
\]
where $0\neq\alpha\in\tk_\kappa$. 
A simple computation shows that two such elements $\gamma_\alpha^\kappa$ and 
$\gamma_\beta^\kappa$ 
are conjugate in $\sigma_\kappa^{-1}\Gamma\sigma_{\kappa}$ if and only if
$\alpha = \eps\beta $ for some $\eps\in\Lambda$ (see also \cite{efrat},
section III.2). Hence summation over parabolic classes means a double summation
over the elements of $\mathcal{S}$ and 
the non-zero elements of $\mathbf{t}_\kappa/\Lambda$.
Therefore the contribution of the parabolic classes in the trace
can be written
as
\begin{align}\label{parabolic_contribution_early_form}
&\sum_{\kappa\in\mathcal{S}}\,
\sum_{0\neq\alpha\in \tk_\kappa/\Lambda}
\sum_{
\begin{array}{c}
\scriptstyle{
\sigma\in C(\gamma) \setminus\Gamma
}\\
\scriptstyle{
\gamma\sim\gamma_{\alpha}^\kappa
}
\end{array}
}
\int_{F_A}
k(z,\sigma^{-1}\gamma\sigma z)u(z)\,d\mu(z)=\nonumber\\[-2mm]
&\qquad
=\sum_{\kappa\in\mathcal{S}}\,
\sum_{0\neq\alpha\in\tk_\kappa/\Lambda}\,
\sum_{\sigma\in C(\gamma_{\alpha}^\kappa)\setminus \sigma_\kappa^{-1}\Gamma\sigma_\kappa}
\int_{\sigma (\sigma_\kappa^{-1} F_A) }k(z,\gamma_{\alpha}^\kappa z)
u(\sigma_\kappa z)\,d\mu(z)
\end{align}
where $C(\gamma_{\alpha}^\kappa)$ is the 
centralizer of $\gamma_{\alpha}^\kappa$ in $\sigma_\kappa^{-1}\Gamma\sigma_\kappa$ given by
\[
C(\gamma_\alpha^\kappa):=\left\{\mtx{1}{\beta}{0}{1}\in PSL(2,\R)^n:\, \beta\in\tk_\kappa\right\},
\]
and its fundamental domain is
$F_{C(\gamma_\alpha^\kappa)}=
\left\{z\in\HH^n:\,0\leq X_1^{\kappa}(z),\dots,X_n^{\kappa}(z)<1\right\}$.

The union of the sets $\sigma (\sigma_\kappa^{-1} F_A)$ in 
(\ref{parabolic_contribution_early_form}) makes up the set
$F_{C({\gamma_{\alpha}^\kappa})}$ 
except for the
images of the part $F\setminus  F_A =F^*_A$. As in the hyperbolic-parabolic case,
for some
cosets the images of $F_A^*$ can be added to the domain we integrate over because 
the kernel function vanishes on those sets.
If $\sigma\in\sigma_\kappa^{-1} \Gamma\sigma_\kappa$ leaves $\infty$ fixed, then 
so does every element in its coset, and
the part 
$\sigma (\sigma_\kappa^{-1}F^*_A)$ is the same as 
\[
\{z\in F_{C({\gamma_{\alpha}^\kappa})}:
        \,\sigma^{-1}z\in   \sigma_\kappa^{-1}\sigma_{\kappa'}U_A,\, 
            \textrm{ for some cusp }\kappa'\in\mathcal{S} \},
\]
at least if $A$ is big enough.
If 
$\kappa\neq\kappa'$, 
then 
$\sigma_{\kappa}^{-1}\sigma_{\kappa'}=\mtx{a}{b}{c}{d}$
does not fix the point $\infty$ and hence $c\neq0$.
Since $\sigma^{-1}\infty = \infty$, the values 
$Y_0(\sigma^{-1}z)$ and $Y_0(z)$ are the same. 
Therefore, if $\sigma^{-1}z\in\sigma_\kappa^{-1}\sigma_{\kappa'} U_A$, then
there is a $w\in U_A$ such that $\sigma^{-1} z=\sigma_\kappa^{-1}\sigma_{\kappa'}w$, and hence
\[
Y_0(z)=Y_0(\sigma^{-1}z)=Y_0(\sigma_\kappa^{-1}\sigma_{\kappa'}w)\leq \frac{1}{(Nc)^2A}.
\]
The inequality above follows easily from the identity 
$\abs{c_kz+d_k}^2\cdot\im (\sigma_\kappa^{-1}\sigma_{\kappa'} z)_k=\im z_k$
that holds for every $k=1,\dots,n$. 
The function $\psi$ is compactly supported, 
hence for a large enough $A$ the kernel
\[
k(z,\gamma_\alpha^\kappa z) = 
\psi\left(\frac{\abs{z_1-(z_1-\alpha_{1})}^2}{y_1^2},\dots,
\frac{\abs{z_n-(z_n-\alpha_{n})}^2}{y_n^2}\right)=
\psi\left(\frac{\alpha_{1}^2}{y_1^2},\dots, \frac{\alpha_{n}^2}{y_n^2}\right)
\]
vanishes for  every $z\in F_{C(\gamma_\alpha)}$ for which $\sigma^{-1}z\in\sigma_\kappa^{-1}\sigma_{\kappa'} U_A$ holds for some
$\kappa\neq\kappa'$,
hence these parts can be simply added to the domain we integrate over.

Now assume that $\sigma$ does not fix the cusp $\infty$. 
Then $\sigma \sigma_\kappa^{-1}\sigma_{\kappa'}$ cannot fix $\infty$, because this would be
equivalent to 
\[
\kappa=(\sigma_\kappa\sigma\sigma_\kappa^{-1}\sigma_{\kappa'})\infty=
(\sigma_\kappa\sigma\sigma_\kappa^{-1})\kappa'.
\]
But $\sigma_\kappa\sigma\sigma_\kappa^{-1}\in\Gamma$ and hence $\kappa\neq\kappa'$ cannot hold, because these two cusps are not equivalent. 
Also, $\sigma_\kappa\sigma\sigma_\kappa^{-1}$
 does not fix $\kappa$, so 
(similarly as above) the parts 
$\sigma\sigma_\kappa^{-1}\sigma_{\kappa'} U_A$ can
be added.

Hence
(\ref{parabolic_contribution_early_form}) becomes
\begin{align*}
&
\sum_{\kappa\in\mathcal{S}}\,
\sum_{0\neq\alpha\in\tk_\kappa/\Lambda}\,\,
\int\limits_{z\in F_{C(\gamma_\alpha^\kappa)}, Y_0(z)\leq A }
\psi\left(\frac{\alpha_{1}^2}{y_1^2},\dots, \frac{\alpha_{n}^2}{y_n^2}\right)
u(\sigma_\kappa z)
\,d\mu(z)=\\[3mm]
&=
\sum_{\kappa\in\mathcal{S}}\,
\sum_{0\neq\alpha\in\tk_\kappa/\Lambda}\,\,
\int\limits_{0\leq X_1^{\kappa},\dots,X_n^\kappa<1 }\,\,\,
\int\limits_{Y_0\leq A}
\psi\left(\frac{\alpha_{1}^2}{y_1^2},\dots, \frac{\alpha_{n}^2}{y_n^2}\right)
u(\sigma_\kappa  z)
\frac{dy_1\ \dots dy_n}{y_1^2\dots y_n^2}
\,dx_1\dots dx_n.
\end{align*}
Using the Fourier expansion of $u(\sigma_\kappa z)$
and that  for an $l\in \tk_{\kappa}^*$ we have
\[
\int\limits_{0\leq X_1^\kappa,\dots, X_n^\kappa<1}e^{2\pi i<l,x>}\,dx_1\dots dx_n =
\left\{
\begin{array}{ll}
     \mathrm{vol}(\R^n/\tk_\kappa),&  \textrm{if }l=0,\\[1mm]
     0 & \textrm{otherwise}, 
\end{array}
\right.
\]
the sum above can be written in the following way:
\[
\sum_{\kappa\in\mathcal{S}}\,
\mathrm{vol}(\R^n/\tk_\kappa)
\sum_{0\neq\alpha\in\tk_\kappa/\Lambda}\,\,
\int\limits_{Y_0\leq A}
\psi\left(\frac{\alpha_{1}^2}{y_1^2},\dots, \frac{\alpha_{n}^2}{y_n^2}\right)
(\eta_\kappa y_1^{s_1}\dots y_n^{s_n}+\phi_\kappa y_1^{1-s_1}\dots y_n^{1-s_n})
\frac{dy_1\dots dy_n}{y_1^2\dots y_n^2}.
\]
The substitution $u_k=\abs{\alpha_{k}}/y_k$  gives then
\begin{align*}
\sum_{\kappa\in\mathcal{S}}\,
\mathrm{vol}(\R^n/\tk_\kappa)
&\sum_{0\neq\alpha\in\tk_\kappa/\Lambda}\,\,
\frac{1}{\abs{N\alpha }}
\int\limits_{\begin{array}{c}\scriptstyle{ 0<u_1,\dots, u_n<\infty}
                                \\\scriptstyle{\abs{N\alpha }\leq Au_1\dots u_n}  
            \end{array}}
\psi\left(u_1^2, \dots, u_n^2\right)\times\\
&\qquad\times\left[\eta_\kappa
\lambda_{m_{u}}(\abs{\alpha})
\frac{\abs{N\alpha}^s}{u_1^{s_1}\dots u_n^{s_n}}
+
\phi_\kappa
\lambda_{-m_{u}}(\abs{\alpha})
\frac{\abs{N\alpha}^{1-s}}{u_1^{1-s_1}\dots u_n^{1-s_n}}
\right]
\,du_1\dots du_n
\end{align*}
where $\abs{\alpha}$ denotes the coordinate-wise absolute value of the vector $\alpha$. 
Hence we have to examine two terms:
\begin{align}\label{2_3_frst_int}
\eta_\kappa \mathrm{vol}(\R^n/\tk_\kappa)\int\displaylimits_{ 0<u_1,\dots, u_n<\infty}
\psi\left(u_1^2,\dots, u_n^2\right)
u_1^{-s_1}\dots u_n^{-s_n}\!\!\!
\sum_{\begin{array}{c}\scriptstyle{0\neq\alpha\in \tk_\kappa/\Lambda}
                    \\\scriptstyle{\abs{N\alpha}\leq Au_1\dots u_n}
            \end{array}
    }
\frac{\lambda_{m_{u}}(\abs{\alpha})}{\abs{N(\alpha)}^{1-s}}
\,\,du_1\dots du_n
\end{align}
and
\begin{align}\label{2_3_scnd_int}
\phi_\kappa \mathrm{vol}(\R^n/\tk_\kappa)\int\displaylimits_{0<u_1,\dots,u_n<\infty}
\psi\left(u_1^2,\dots, u_n^2\right)
u_1^{s_1-1}\dots u_n^{s_n-1}\!\!\!
\sum_{\begin{array}{c}
            \scriptstyle{0\neq\alpha\in \tk_\kappa/\Lambda}   \\
            \scriptstyle{\abs{N\alpha}\leq Au_1 \dots u_n}
        \end{array}
    }
\frac{\lambda_{-m_{u}}(\abs{\alpha})}
{\abs{N(\alpha)}^{s}}
\,\,du_1\dots d u_n.
\end{align}

We express the sums in (\ref{2_3_frst_int}) and (\ref{2_3_scnd_int})
in terms of the zeta functions $Z_\kappa(1-s,-m_u)$ and $Z_\kappa(s,m_u)$, respectively.
By Lemma \ref{zeta_lemma} $Z_\kappa(s,m_u)$ can be 
continued meromorphically to $\C$ 
with simple 
poles at $0$ and $1$  if and only if $m_u=0$, 
and in the latter case its residue at $1$ is 
$\frac{2^n|\det\mathcal{E}|}{n\cdot \mathrm{vol}(\R^n/\tk_\kappa)}$.
By Proposition \ref{zeta_prop}
there is a vector $\nu\in\R^n$ with non-zero coordinates
 such that the coordinates of $\nu\cdot\tk_\kappa$ are
conjugate integers. Hence we may write
\begin{align*}
Z_\kappa(s,m_u)&=
\sum_{0\neq\alpha\in \nu\cdot\tk_\kappa/\Lambda}
\frac{\lambda_{-m_u} (\abs{\alpha}/\nu)}
{\abs{N(\alpha/\nu)}^{s}}
&=
|N\nu|^{s}\sum_{0\neq\alpha\in \nu\cdot\tk_\kappa/\Lambda}
\frac{\lambda_{-m_u} (\abs{\alpha/\nu})}
{\abs{N(\alpha)}^{s}}
=|N\nu|^{s}
\sum_{ k=1}^\infty
\frac{a_{m_u}(k)}
{k^{s}},
\end{align*}
where 
\[
a_{m_u}(k)=\sum_{0\neq \alpha\in \nu\cdot\tk_\kappa/\Lambda,\, \abs{N\alpha}=k}
\lambda_{-m_u} (\abs{\alpha}/\nu).
\]
Since $\Lambda$ is isomorphic to a finite index subgroup of the multiplicative group of 
the units in $\OO_K$,
the latter sum can be estimated from above by the number of integer ideals of norm $k$ in $K$ 
and hence by $\tau(k)^{[K:\Q]}\ll_\delta k^{\delta}$ for any $\delta>0$, where
$\tau(k)$ is the number of divisors of the rational integer $k$.

Now we can  apply Theorem 5.2 and Corollary 5.3 in \cite{montgomery_vaughan} for the function
\[
\alpha_{s,m_u}(S)=\sum_{k=1}^\infty\frac{a_{m_u}(k)}{k^{s+S}}=\frac{Z_\kappa(s+S,m_u)}{|N\nu|^{s+S}}.
\]
If $0<\ree s<1$ and 
 $\sigma_0>1-\ree s$, then
\begin{align*}
\sum_{k\leq A}\mathop{}^{\mkern-5mu '} \frac{a_{m_u}(k)}{k^s} = 
\frac{1}{2\pi i}
\int\limits_{\sigma_0-iT}^{\sigma_0+iT}\alpha_{s,m_u}(S)\frac{A^S}{S}\,d S+R_{s,m_u}
\end{align*}
where $\sum\mathop{}^{\mkern-5mu '}$ indicates that if $A$ is an integer, then the last term
is to be counted with half weight, further
\[
R_{s,m_u}\ll \sum_{
        A/2<k < 2A,\, k\neq A
}
\abs{a_{m_u}(k)}k^{-\ree s}\min\left(1,\frac{A}{T\abs{A-k}}\right)+
\frac{4^{\sigma_0}
+A^{\sigma_0}}{T}\sum_{k = 1}^\infty\frac{\abs{a_{m_u}(k)}}{k^{\sigma_0+\ree s}}.
\]
Hence for any $\sigma_0>\ree s$
the integral in (\ref{2_3_frst_int})
can be rewritten as
\begin{align*}
|N\nu|^{1-s}&
\!\!\!\!\int\displaylimits_{ 0<u_1,\dots, u_n<\infty}
\psi\left(u_1^2,\dots, u_n^2\right)
u_1^{-s_1}\dots u_n^{-s_n}\times\\
&\,\,\qquad\times\left[\frac{1}{2\pi i}
\int\limits_{\sigma_0-iT}^{\sigma_0+iT}\!\!
\alpha_{1-s,-m_u}(S)
\cdot\frac{(|N\nu|Au_1\dots u_n)^S}{S}\,d S+R_{1-s,-m_u}\right]
\,du_1\dots du_n,
\end{align*}
where with the notation $B=|N\nu|Au_1\dots u_n$ we have that $R_{1-s,-m_u}$ is bounded by
\begin{align}\label{Rror}
\sum_{
        B/2<k < 2B,\, k\neq B
}
\abs{a_{-m_u}(k)}k^{\ree s-1}\min\left(1,\frac{B}{T\abs{B-k}}\right)
+\frac{4^{\sigma_0}+B^{\sigma_0}}{T}\sum_{k = 1}^\infty\frac{\abs{a_{-m_u}(k)}}{k^{\sigma_0+1-\ree s}}.
\end{align}

Let us fix an $0<\delta_0<1-\ree s$ 
and use the estimate  $a_{-m_u}(k)\ll k^{\delta_0}$. Also, we 
set the values $\sigma_0 = \ree s+\delta_0+\frac{1}{\log A}$  and $T=A^{\ree s+\delta_1}$ for some $\delta_0<\delta_1<1-\ree s$.
Since $u_k$ is bounded from above for any $1\leq k\leq n$,  
the second term on the right hand side of (\ref{Rror}) is bounded by
$A^{\delta_0-\delta_1}\log A=o(1)$ (as $A\to\infty$), and the implied constant depends on $\delta_0$.

Turning to the first term
we
divide the sum in it into three parts.
The first one is where 
$\abs{B-k}<BT$. 
The terms of this part are of the form $\abs{a_{-m_u}(k)}k^{\ree s-1}\ll k^{\delta_0+\ree s-1}$
hence they 
give at most a constant times
$(BT)\cdot B^{\delta_0+\ree s-1}=
(|N\nu|u_1\dots u_n)^{\delta_0+\ree s}A^{\delta_0-\delta_1}
$, and since $\psi$ is compactly supported this gives an $o(1)$ term in the last integral above.

The second part is where $BT\leq \abs{B-k}<BT+1$,
i.e. it consists of at most two terms
bounded by a constant times
\[
k^{\delta_0+\ree s-1}\frac{B}{T\abs{B-k}}\ll\frac{B^{\delta_0+\ree s}}{T^2B}\leq B^{\delta_0+\ree s-1}=
(|N\nu| A u_1\dots u_n)^{\delta_0+\ree s-1}.
\]
Hence we obtain a term in the integral above that can be bounded by
\[
A^{\delta_0+\ree s-1}\int\limits_{0<u_1,\dots,u_n< C} (u_1\dots u_n)^{\delta_0-1}\,du_1\dots du_n
\]
for some $C>0$, and  the latter integral converges at $0$ 
giving an $o(1)$ term as $A\to\infty$.

The third part is where 
$|BT+1|\leq \abs{B-k}$.
Note that in this case
 $1\leq \abs{B-k}\leq B$,
hence this error term is bounded by
\begin{align*}
B^{\delta_0+\ree s-1}\cdot\frac{B}{T}
\sum_{ 1\leq \abs{B-k}\leq  B}\frac{1}{\abs{B-k}}&\ll (|N\nu|u_1\dots u_n)^{\delta_0+\ree s}A^{\delta_0-\delta_1}\sum_{1\leq k\leq B}\frac{1}{k}
\\&
\ll A^{\delta_0-\delta_1}\max(0,\log B)\ll A^{\delta_0-\delta_1}\log A=o(1).
\end{align*}

Finally, to cover also those cases when $k=|N\nu| Au_1\dots u_n$ is an integer we may add
\[
\frac{1}{2}\cdot\frac{a_{-m_u}(k)}{k^{1-s}}\ll  (|N\nu|Au_1\dots u_n)^{\delta_0+\ree s-1}
\]
to the error term $R_{1-s,-m_u}$, which  also gives an $o(1)$ term.

It follows that aside from an $o(1)$ term the expression in (\ref{2_3_frst_int})
is
\begin{align*}
\eta_\kappa \mathrm{vol}(\R^n/\tk_\kappa)|N\nu|^{1-s}\!\!\!\!\!&\int\displaylimits_{ 0<u_1,\dots, u_n<\infty}
\psi\left(u_1^2,\dots, u_n^2\right)
u_1^{-s_1}\dots u_n^{-s_n}\times\\
&\,\qquad\times\left(\frac{1}{2\pi i}
\int\limits_{\sigma_0-iT}^{\sigma_0+iT}\!\!
\alpha_{1-s,-m_u}(S)
\cdot\frac{(|N\nu|Au_1\dots u_n)^S}{S}\,d S\right)
du_1\dots du_n
\end{align*}
where $\sigma_0 = \ree s+\delta_0+\frac{1}{\log A}$
and 
$T=A^{\ree s+\delta_1}$ for some $\delta_0<\delta_1<1-\ree s$.
Substituting the definition of $\alpha_{1-s,-m_u}(S)$ and interchanging the order of integration this becomes
\begin{align*}
    \frac{\eta_\kappa  \mathrm{vol}(\R^n/\tk_\kappa)}{2\pi i}
\int\limits_{\sigma_0-iT}^{\sigma_0+iT}
F(S)
Z_\kappa(1-s+S,-m_{u})\frac{A^S}{S}
\,dS,
\end{align*}
where
\begin{align*}
F(S)=
\int\limits_0^\infty\dots
\int\limits_0^\infty
\psi\left(u_1^2,\dots, u_n^2\right)
u_1^{S-s_1}\dots u_n^{S-s_n}
\,du_1\dots\, du_n.
\end{align*}
Let us choose a number 
$\ree s-1<\sigma_1<0$
and set $G(S)=F(S)Z_\kappa(1-s+S,-m_{u})\frac{A^S}{S}$.
We shift the line of integration to the line $\sigma_1+it$, then by the residue theorem
\begin{align}\label{residue}
    \frac{1}{2\pi i}
    \int\limits_{\sigma_0-iT}^{\sigma_0+iT}G(S)\,d S=&
Z_\kappa(1-s,-m_{u})F(0)+
\delta_{m_u}\cdot
\frac{2^n|\det\mathcal{E}|}{n\cdot \mathrm{vol}(\R^n/\tk_\kappa)}\cdot\frac{A^{s}}{s}
F\left(s\right)\nonumber \\[1mm]
&-\frac{1}{2\pi i}\left(\,\,
\int\limits_{\sigma_0+iT}^{\sigma_1+iT}G(S)\,dS+
\int\limits_{\sigma_1+iT}^{\sigma_1-iT}G(S)\,dS
+\int\limits_{\sigma_1-iT}^{\sigma_0-iT}G(S)\,dS
\right).
\end{align}

We show that the last three integral above is $o(1)$ as $A\to\infty$. 
Firstly, repeated integration by parts in $F(S)$ with respect to $u_1$ (for example)
gives
\[
G(S)=Z_\kappa(1-s+S,-m_{u})\frac{A^S}{S^{N+1}}
\int\limits_0^\infty\dots
\int\limits_0^\infty
H_{s_1}^{(N)}(u_1,\dots, u_n)
u_1^{S-s_1}\dots u_n^{S-s_n}
\,du_1\dots\,du_n
\]
where $N$ is any positive integer and $H_{s_1}^{(N)}$ is a compactly supported smooth function.

To estimate the integrals  on the right hand side of (\ref{residue}) 
we apply Lemma \ref{zeta_lemma}:
if $0\leq \ree s$, then 
we have
$Z_\kappa(s,-m_u)\ll\abs{\im s}^{n/2+\eps}$
for any $\eps>0$  as $\abs{t}\to\infty$,
hence  on the horizontal segments we have
\begin{align*}
Z_\kappa(1-s+S,-m_{u})\frac{(u_1\dots u_nA)^S}{S^{N+1}}
&\ll\frac{A^{\ree S}}{T^{N+1}}\cdot T^{n/2+\eps}
\ll A^{\ree S+(\ree s+\delta_1)(n/2+\eps-N-1)}.
\end{align*}
Here  (by the compact support of $H_{s_1}^{(N)}$) $u_1,\dots, u_n$ 
can be bounded from above by a constant. 
Choosing an appropriate $N$ it follows that these integrals  give $o(1)$ terms.
  On the vertical line we have
\begin{align*}
Z_\kappa(1-s+S,-m_{u})\frac{(u_1\dots u_nA)^S}{S^{N+1}}
&\ll \frac{A^{\sigma_1}}{(1+\abs{\im S})^{N+1}}\cdot
\abs{\im S}^{n/2+\eps}\ll A^{\sigma_1}\abs{\im S}^{n/2+\eps-N-1}
\end{align*}
if $\abs{\im S}$ is big enough and hence (choosing an $N>n/2+\eps-1$)
\begin{align*}
\int\limits_{\sigma_1+iT}^{\sigma_1-iT}
G(S)\,dS&\ll A^{\sigma_1}=o(1).
\end{align*}

The expression (\ref{2_3_frst_int})
becomes 
\[
\eta_\kappa\mathrm{vol}(\R^n/\tk_\kappa)
Z_\kappa(1-s,-m_{u})F(0)+
\delta_{m_u}\cdot\frac{\eta_\kappa 2^n|\det\mathcal{E}|}{n}\cdot\frac{A^{s}}{s}
F\left(s\right)
+o(1)
\]
One can show similarly that (\ref{2_3_scnd_int}) is
\[
\phi_\kappa \mathrm{vol}(\R^n/\tk_\kappa)
Z_\kappa(s,m_{u})\tilde{F}(0)+
\delta_{m_u}\cdot\frac{\phi_\kappa 2^n|\det\mathcal{E}|}{n}\cdot\frac{A^{1-s}}{1-s}
\tilde{F}\left(1-s\right)
+o(1)
\]
where
\begin{align*}
\tilde{F}(S)=
\int\limits_0^\infty\dots
\int\limits_0^\infty
\psi\left(u_1^2,\dots, u_n^2\right)
u_1^{S+s_1-1}\dots u_n^{S+s_n-1}
\,du_1\dots\, du_n.
\end{align*}

Next we calculate the values $2^nF(s)$ and $2^n\tilde{F}(s)$ in the case $m_u=0$, when
\begin{align*}
2^nF(s)=2^n\tilde{F}(1-s)=
&2^n
\int\limits_0^{\infty}
\dots
\int\limits_0^{\infty}
\psi\left(
    u_1^2,\dots,u_n^2
\right)\prod\limits_{k=1}^ndu_k\\
&=
\int\limits_{0}^{\infty}
\dots
\int\limits_{0}^{\infty}
\frac{\psi(t_1,\dots,t_n)}{\prod\limits_{k=1}^n\sqrt{t_k}}
\,\prod_{k=1}^n\,dt_k
=
g(0,\dots,0).
\end{align*}

Finally, we evaluate $F$ and $\tilde{F}$ at $0$:
\begin{align*}
F(0)&=
\int\limits_0^\infty\dots
\int\limits_0^\infty
\psi(u_1^2,\dots, u_n^2)
u_1^{-s_1}\dots u_n^{-s_n}
\,du_1\dots \,du_n
\\[3mm]
&
=
\frac{1}{2^n}
\int\limits_0^\infty\dots
\int\limits_0^\infty
\frac{\psi(u_1,\dots, u_n)}{\sqrt{u_1\dots u_n}}
u_1^{-\frac{s_1}{2}}\dots u_n^{-\frac{s_n}{2}}
\,du_1\dots \,du_n.
\end{align*}
Using (\ref{2_1_transform_rules}) we get that this is
\begin{align*}
&\frac{(-1)^n}{(2\pi)^n}
\int\limits_{0}^\infty\dots
\int\limits_{0}^\infty
\left(\,
\int\limits_{u_n}^\infty\dots
\int\limits_{u_1}^\infty
\frac{\frac{\partial^n Q}{\partial w_1\dots \partial w_n}(w_1,\dots,w_n)}
     {\sqrt{w_1-u_1}\dots \sqrt{w_n-u_n}}
\,dw_1\dots \,dw_n\right)
\frac{u_1^{-\frac{s_1}{2}}\dots u_n^{-\frac{s_n}{2}}}{\sqrt{u_1\dots u_n}}
\,du_1\dots \,du_n\\[3mm]
&=
\frac{(-1)^n}{(2\pi)^n}
\int\limits_0^\infty\dots
\int\limits_0^\infty
\frac{\partial^n Q}{\partial w_1\dots \partial w_n}(w_1,\dots,w_n)
\left(
\prod_{k=1}^n\int\limits_{0}^{w_k} 
         \frac{u_k^{-\frac{s_k}{2}}}{\sqrt{w_k-u_k}\sqrt{u_k}}
\,du_k
\right)
\,dw_1\dots \,dw_n.
\end{align*}
We have
\[
\int\limits_0^{w}\frac{u^{-\frac{\alpha}{2}}}{\sqrt{w-u}\sqrt{u}}\,du = 
w^{-\frac{\alpha}{2}}
B\left(\frac{1-\alpha}{2},\frac{1}{2}\right) = 
w^{-\frac{\alpha}{2}}
\frac{\Gamma(\frac{1-\alpha}{2})\Gamma(\frac{1}{2})}{\Gamma(\frac{1-\alpha}{2}+\frac{1}{2})}
=w^{-\frac{\alpha}{2}}2^{-\alpha}\frac{\Gamma(\frac{1-\alpha}{2})^2}{\Gamma(1-\alpha)}
\]
for any $w>0$ and $\alpha\in\C$ with $0<\ree \alpha<1$, where $B$ is the beta function and we 
used the following relations:
\[
\Gamma\left(1/2\right)=\sqrt{\pi},\qquad 
\Gamma(z)\Gamma\left(z+\frac{1}{2}\right) = 
2^{1-2z}\sqrt{\pi}\Gamma(2z).
\]
Then $F(0)$ is
\[
\frac{(-1)^n}{(2^{1-s}\pi)^n}
\left(\prod_{k=1}^n
\frac{\Gamma(\frac{1-s_k}{2})^2}{\Gamma(1-s_k)}
\right)
\int\limits_0^\infty\dots
\int\limits_0^\infty
\frac{\partial^n Q}{\partial w_1\dots \partial w_n}(w_1,\dots,w_n)
w_1^{-\frac{s_1}{2}}\dots
w_n^{-\frac{s_n}{2}}
\,dw_1\dots \,dw_n.
\]
We substitute $w_k= e^{x_k}+e^{-x_k}-2$ to express this in terms of the function $g$,
and then the integral above becomes
\[
\int\limits_0^\infty\dots 
\int\limits_0^\infty
\frac{\partial^n g}{\partial x_1\dots \partial x_n}(x_1,\dots, x_n)
\prod_{k=1}^n(e^{x_k}+e^{-x_k}-2)^{-\frac{s_k}{2}}
\,dx_k.
\]
Now
\[
\frac{\partial^n g}{\partial x_1\dots \partial x_n}(x_1,\dots, x_n)
=\frac{(-i)^n}{(2\pi)^n}
\int\limits_0^\infty\dots 
\int\limits_0^\infty
h(r_1,\dots, r_n)r_1\dots r_ne^{-i(r_1x_1+\dots +r_nx_n)}
\,dr_1\dots \,dr_n,
\]
and
\begin{align*}
\int\limits_0^\infty
(e^{x}+e^{-x}-2)^{-\alpha}e^{-irx}
\,dx
&=
\int\limits_0^1
\left(y^{-\frac{1}{2}}-y^{\frac{1}{2}}\right)^{-2\alpha}
y^{ir-1}
\,dy
=
\int\limits_0^1
\left(1-y\right)^{-2\alpha}
y^{\alpha+ir-1}
\,dy\\[3mm]
&=B(\alpha+ir,1-2\alpha)=\frac{\Gamma(\alpha+ir)\Gamma(1-2\alpha)}{\Gamma(1-\alpha+ir)}.
\end{align*}
We conclude that $F(0)$ is
\[
\left(\frac{i}{2^{2-s}\pi^2}\right)^n
\left(
    \prod_{k=1}^n
        \Gamma\left(\frac{1-s_k}{2}\right)^2
\right)
\int\limits_0^\infty\dots
\int\limits_0^\infty
h(r_1,\dots, r_n)\prod_{k=1}^n 
\frac{r_k\Gamma(\frac{s_k}{2}+ir_k)}{\Gamma(\frac{2-s_k}{2}+ir_k)}
\,dr_k.
\]
Similarly, $\tilde{F}(0)$ is
\[
\left(\frac{i}{2^{s+1}\pi^2}\right)^n
\left(
    \prod_{k=1}^n
        \Gamma\left(\frac{s_k}{2}\right)^2
\right)
\int\limits_0^\infty\dots
\int\limits_0^\infty
h(r_1,\dots, r_n)\prod_{k=1}^n 
\frac{r_k\Gamma(\frac{1-s_k}{2}+ir_k)}{\Gamma(\frac{s_k+1}{2}+ir_k)}
\,dr_k,
\]
and this completes the proof of Theorem \ref{parabolic_thm}.

\bibliographystyle{abbrv}
\bibliography{references}

\end{document}